\documentclass[11pt,reqno]{amsart}
\usepackage{fullpage}

\usepackage{graphicx}
\usepackage{hyperref}
\usepackage{amsmath,amsopn,amssymb,amsfonts,stmaryrd}
\usepackage{verbatim}
\usepackage{amsthm}
\usepackage{mathtools}
\usepackage{color}
\usepackage{enumitem}
\usepackage[framemethod=TikZ]{mdframed}
\usepackage{bbm}
\usepackage{mathrsfs}
\usepackage{booktabs}
\usepackage{caption}
\usepackage{bm}
\usepackage{tensor}
\usepackage{cancel}

\usepackage{mathrsfs}
\usepackage[scr=boondox]{mathalfa}

\newcommand{\R}{{\mathbb R}}
\newcommand{\C}{{\mathbb C}}
\newcommand{\Z}{{\mathbb Z}}
\newcommand{\N}{{\mathbb N}}
\newcommand{\Q}{{\mathbb Q}}
\renewcommand{\H}{{\mathbb H}}

\newcommand{\calG}{\mathcal{G}}

\newcommand{\calP}{\mathcal{P}}
\newcommand{\calQ}{\mathcal{Q}}

\newcommand{\calS}{\mathcal{S}}

\newcommand{\calU}{\mathcal{U}}

\renewcommand{\calS}{\mathcal{S}}

\newcommand{\arctanh}{\operatorname{arctanh}}

\newcommand{\sgn}{\mbox{sgn}}

\newcommand{\SL}{\mathrm{SL}}

\newcommand{\GL}{\mathrm{GL}}

\newcommand{\ol}[1]{\overline{#1}}

\theoremstyle{plain}
\newtheorem{thm}{Theorem}[section]

\newtheorem{lem}[thm]{Lemma}
\newtheorem{lemma}[thm]{Lemma}
\newtheorem{prop}[thm]{Proposition}

\theoremstyle{definition}
\newtheorem{defn}[thm]{Definition}
\newtheorem{rem}[thm]{Remark}

\numberwithin{equation}{section}

\newcommand{\pmat}[1]{\left( \smallmatrix #1 \endsmallmatrix \right)}

\newcommand{\mat}[1]{\left( \begin{matrix} #1 \end{matrix} \right)}

\renewcommand{\sgn}{\textnormal{sgn}}

\def\lp{\left(}
\def\rp{\right)}

\def\a{\alpha}
\def\b{\beta}
\def\d{\delta}

\def\l{\lambda}

\def\z{\zeta}

\def\n{\nu}
\def\s{\sigma}
\def\g{\gamma}
\def\t{\tau}

\def\OO{{\Omega}}
\def\GG{\Gamma}

\def\LL{\Lambda}

\def\del{  \partial}

\allowdisplaybreaks

\renewcommand{\sgn}{{\rm sgn}}

\def\wh{\widehat}
\def\bar{\overline}
\newcommand{\andd}{\quad \mbox{ and } \quad}

\setlist[itemize]{noitemsep, topsep=0pt}

\makeatletter
\newcommand{\vast}{\bBigg@{2}}
\newcommand{\Vast}{\bBigg@{5}}
\makeatother

\renewcommand{\pmod}[1]{\    \left(  \mathrm{mod} \,  #1 \right)}

\makeatletter
\newcommand{\nolisttopbreak}{\par\nobreak\@afterheading} 
\makeatother

\newcommand{\mT}{\Theta}
\newcommand{\mt}{\vartheta}

\newcommand{\mF}{F}
\newcommand{\mf}{f}
\newcommand{\mcalF}{\mathcal{F}}
\newcommand{\mfrakf}{\mathfrak{f}}

\newcommand{\mG}{G}
\newcommand{\mg}{g}
\newcommand{\mcalG}{\mathcal{G}}
\newcommand{\mfrakg}{\mathfrak{g}}

\newcommand{\mH}{H}
\newcommand{\mh}{h}
\newcommand{\mcalH}{\mathcal{H}}
\newcommand{\mfrakh}{\mathfrak{h}}

\title{Quantum Modular Forms from Real-Quadratic Double Sums}
\author{Kathrin Bringmann}
\author{Caner Nazaroglu}
\address{University of Cologne, Department of Mathematics and Computer Science, Weyertal 86-90, 50931 Cologne, Germany}
\email{kbringma@math.uni-koeln.de}
\email{cnazarog@math.uni-koeln.de}

\begin{document}

\begin{abstract}
In 2015, Lovejoy and Osburn discovered twelve $q$-hypergeometric series and proved that their Fourier coefficients can be understood as counting functions of ideals in certain quadratic fields. In this paper, we study their modular and quantum modular properties and show that they yield three vector-valued quantum modular forms on the group $\GG_0 (2)$.
\end{abstract}

\maketitle

\section{Introduction and statement of results}
Starting with the work of Andrews, Dyson, and Hickerson \cite{ADH} and of Cohen \cite{Co}, a surprising interplay has been uncovered between $q$-hypergeometric series, real-quadratic fields, and classical Maass forms. They discovered this interplay by exploring the now famous function
\begin{equation}\label{eq:sigma_q_hypergeometric}
\s(q) := \sum_{n=0}^\infty \frac{q^{\frac{n(n+1)}{2}}}{(-q;q)_n},
\end{equation}
which was first studied by Ramanujan in his lost notebook \cite{Ram}. Here and throughout $(a;q)_n=(a)_n:=\prod_{j=0}^{n-1}(1-aq^j)$ for $n\in\N_0\cup\{\infty\}$. A key step in the analysis of \cite{ADH} is the use of Bailey chains to rewrite $\s(q)$ as a ``false-indefinite theta function":
\begin{equation*}
\s(q) = \underset{-n\le j\le n}{\sum_{n\geq 0}} (-1)^{n+j}  \left(1-q^{2n+1}\right)q^{\frac{n(3n+1)}{2}-j^2}.
\end{equation*}
This representation allowed the authors of \cite{ADH} to relate the Fourier coefficients of $\s(q)$ to the arithmetic of the real-quadratic field $\Q(\sqrt{6})$ and to prove many interesting properties for these Fourier coefficients, which are not obvious from the combinatorial interpretation of the $q$-hypergeometric series in equation \eqref{eq:sigma_q_hypergeometric}. Through this arithmetic interpretation, it was also possible to extend the Fourier coefficients of $\s (q)$ in
\begin{equation*}
\s (q) =  \sum_{n=0}^\infty T(24n+1)  q^n,
\end{equation*}
to a sequence $T(24n+1)$ defined for all $n \in \Z$. This led the authors of \cite{ADH} and \cite{Co} to a partner $q$-series
\begin{equation*}
\s^* (q) := \sum_{n=1}^\infty T(1-24n) q^n
=
2 \sum_{\substack{m \geq 0 \\ 2k \geq 3m+1}} (-1)^{m+k} q^{k^2 - \frac{m(3m+1)}{2}}
\lp 1 + q^{2(k-m)} \rp,
\end{equation*}
which also has a representation as a $q$-hypergeometric series:
\begin{equation*}
\s^*(q) = 2\sum_{n=1}^\infty \frac{(-1)^nq^{n^2}}{\left(q;q^2\right)_n}.
\end{equation*}
Now the remarkable fact discovered by \cite{Co} is that the Fourier coefficients of $\s(q)$ and $\s^* (q)$ are those of a Maass form defined as
\begin{equation*}
u(\t) := \sqrt{\t_2}\sum_{n\in\Z+\frac{1}{24}} T(24n)K_0(2\pi|n|\t_2) e^{2\pi in\t_1} \quad \mbox{for } \t = \t_1 + i\t_2 \in \H,
\end{equation*}
where $K_\nu$ denotes the $K$-Bessel function of order $\nu$. As reviewed in Section \ref{sec:maass_forms}, Maass forms are invariant under modular transformations and this forms the number theoretical aspect of the three-piece interplay with the combinatorial and algebraic aspects mentioned above.

Building upon these results, Zagier showed in \cite{Za} that\footnote{
The expressions for $\s (q)$ and $\s^* (q)$ as $q$-hypergeometric series are related to each other under the transformation $q \mapsto q^{-1}$ as noted by Cohen (see \cite{Co}). This relation was employed to define the function $f(x)$ through their common limits to the roots of unity as shown.
}
\begin{equation*}
f(x) := e^\frac{\pi i x}{12} \lim_{t \to 0^+} \sigma\left(e^{2\pi i (x+it)}\right)
= -e^{\frac{\pi ix}{12}} \lim_{t \to 0^+} \sigma^*\left(e^{-2\pi i (x-it)}\right)
\end{equation*}
is a so-called quantum modular form, which further elaborates the relation discussed above.
Recall that in the simplest case, {\it a quantum modular form} of weight $k$ and with quantum set $\calQ \subset \Q$ is a function $g:\calQ\to\C$ whose {\it obstruction to modularity}
\begin{equation*}
g(x)-(cx+d)^{-k} g\left(\frac{ax+b}{cx+d}\right) 
\quad \mbox{for }
\mat{a&b\\c&d} \in \GG \subset \SL_2(\Z) 
\end{equation*}
is analytically ``nice", e.g.~it extends real-analytically to $\R \setminus \{-\frac{d}{c}\}$. More generally, one can discuss vector-valued generalizations with nontrivial multiplier systems (as in Theorem \ref{thm:frak_uj_quantum_modular_transformations}) and require stronger analytic properties from the obstructions to modularity such as holomorphicity (as in Proposition \ref{prop:modularity_error_holomorphy} and Remark \ref{rem:error_modularity_holomorphic_extension_cut_plane}).

At this point we should note that the discussion we have had so far is not a peculiar property of the functions $\s(q)$ and $\s^* (q)$.
Since the work of \cite{ADH} and \cite{Co}, a number of generalizations have been investigated in \cite{BK,BLR,CFLZ,Lo}. In fact, the main focus of this paper is on such a generalization developed by Lovejoy and Osburn \cite{LoOs}. They studied twelve $q$-hypergeometric series defined as 
\begin{align*}
L_1(q) &:= \sum_{1\le k\le n} \frac{(-1)^{n+k} (q)_{n-1} q^{\frac{n(n+1)}{2}+\frac{k(k+1)}{2}}}{\left(1-q^{2k-1}\right) (q)_{n-k} (q)_{k-1}},
&& \!\!
L_2(q) := \sum_{0\le k\le n} \frac{(-1)^{n+k} (q)_n q^{\frac{n(n+1)}{2}+\frac{k(k+1)}{2}}}{\left(1-q^{2k-1}\right) (q)_{n-k} (q)_k},\\
L_3(q) &:= q\sum_{1\le k\le n} \frac{(-1)^{n+k} (q)_{n-1} q^{\frac{n(n+1)}{2}+\frac{k(k-1)}{2}}}{\left(1-q^{2k-1}\right) (q)_{n-k} (q)_{k-1}},
&&\!\!
L_4(q) := -1 + \sum_{0\le k\le n} \frac{(-1)^{n+k} (q)_n q^{\frac{n(n+1)}{2}+\frac{k(k-1)}{2}}}{\left(1-q^{2k+1}\right) (q)_{n-k} (q)_k}, 
\\
L_5(q) &:= q\sum_{1\le k\le n} \frac{(-1)^{n+k}(-1)_n(q)_{n-1}q^{n+k^2-k}}{\left(1-q^{2k-1}\right) (q)_{n-k}\left(q^2;q^2\right)_{k-1}},
&&\!\!
L_6(q) := \sum_{1\le k\le n} \frac{(-1)^{n+k}(-1)_n(q)_{n-1}q^{n+k^2}}{\left(1-q^{2k-1}\right) (q)_{n-k}\left(q^2;q^2\right)_{k-1}},\\
L_7(q) &:= 2\sum_{0\le k\le n}\hspace{-.25cm}{\vphantom{\sum}}^* \frac{(-1)^{n+k}\left(q^2;q^2\right)_nq^{k^2+k}}{\left(1-q^{2k+1}\right) (q)_{n-k}\left(q^2;q^2\right)_k},
&&\!\!
L_8(q) := -1+2\sum_{0\le k\le n}\hspace{-.25cm}{\vphantom{\sum}}^* \frac{(-1)^{n+k}\left(q^2;q^2\right)_n q^{k^2}}{\left(1-q^{2k+1}\right)(q)_{n-k}\left(q^2;q^2\right)_k},
\\
L_9(q) &:= \sum_{1\le k\le n} \frac{(-1)^{n+k}(-1)_n(q)_{n-1} q^{n+\frac{k(k+1)}{2}}}{\left(1-q^{2k-1}\right)(q)_{n-k}(q)_{k-1}},
&&\!\!
L_{10}(q) := q\sum_{1\le k\le n} \frac{(-1)^{n+k}(-1)_n(q)_{n-1} q^{n+\frac{k(k-1)}{2}}}{\left(1-q^{2k-1}\right)(q)_{n-k}(q)_{k-1}},\\
L_{11}(q) &:= 2\sum_{0\le k\le n}\hspace{-.25cm}{\vphantom{\sum}}^* \frac{(-1)^{n+k}\left(q^2;q^2\right)_nq^{\frac{k(k+1)}{2}}}{\left(1-q^{2k+1}\right) (q)_{n-k}(q)_k},
&&\!\!
L_{12}(q) := -2+2\sum_{0\le k\le n}\hspace{-.25cm}{\vphantom{\sum}}^* \frac{(-1)^{n+k} \left(q^2;q^2\right)_nq^{\frac{k(k-1)}{2}}}{\left(1-q^{2k+1}\right)(q)_{n-k}(q)_k},
\end{align*}
where the symbol * indicates that we take the average of even and odd partial sums (in $n$) to obtain convergence. They then showed that these functions are also related to the arithmetic of real-quadratic fields. 
In particular, $L_1,\ldots,L_4$ count ideals in the ring of integers $\mathcal{O}_{\Q (\sqrt{2})}$ of the real-quadratic field $\Q (\sqrt{2})$, $L_5,\ldots,L_8$ count ideals in $\mathcal{O}_{\Q (\sqrt{3})}$, and  $L_9,\ldots,L_{12}$ count ideals in $\mathcal{O}_{\Q (\sqrt{6})}$. In this work, we investigate the modular aspect of these functions and prove the following result.

\begin{thm}\label{thm:main_result}
The limits 
\begin{equation*}
\lim_{t \to 0^+} \mat{
e^{-\frac{17 \pi i x}{16}} L_1\left(e^{2 \pi i \left(x+it\right)}\right) \\
e^{\frac{7 \pi i x}{16}} L_2\left(e^{2 \pi i \left(x+it\right)}\right) \\
e^{-\frac{33 \pi i x}{16}} L_3\left(e^{2 \pi i \left(x+it\right)}\right) \\
e^{-\frac{9 \pi i x}{16}} L_4\left(e^{2 \pi i \left(x+it\right)}\right)
}
,\qquad
\lim_{t \to 0^+} \mat{
e^{-\frac{9 \pi i x}{8}} L_9\left(e^{2 \pi i \left(x+it\right)}\right) \\
e^{-\frac{17 \pi i x}{8}} L_{10}\left(e^{2 \pi i \left(x+it\right)}\right) \\
e^{\frac{5 \pi i x}{24}} L_{11}\left(e^{2 \pi i \left(x+it\right)}\right) \\
e^{-\frac{19 \pi i x}{24}} L_{12}\left(e^{2 \pi i \left(x+it\right)}\right) 
},
\end{equation*}
as well as the finite part of 
\begin{equation*}
\mat{
e^{-2 \pi i x} L_5\left(e^{2 \pi i \left(x+it\right)}\right) - \frac{2}{\pi} \arctanh \lp \frac{1}{\sqrt{3}} \rp
\\
e^{- \pi i x} L_6\left(e^{2 \pi i \left(x+it\right)}\right) \\
e^{\frac{\pi i x}{3}} L_7\left(e^{2 \pi i \left(x+it\right)}\right) \\
e^{-\frac{2 \pi i x}{3}} L_8\left(e^{2 \pi i \left(x+it\right)}\right)
}
\quad \mbox{as } t \to 0^+
\end{equation*}
all form vector-valued quantum modular forms over the group $\GG_0 (2)$.
\end{thm}
More precise statements are given in Propositions \ref{prop:L1_4_quantum_modular}, \ref{prop:L5_8_quantum_modular}, and \ref{prop:L9_12_quantum_modular}. The quantum modular forms $\mfrakf_j$, $\mfrakg_j$, and $\mfrakh_j$ referenced in these propositions are related to the functions $L_j$ given above by Lemmas \ref{lem:l1l4rew}, \ref{lem:l5l8rew}, \ref{lem:l9l12rew} and equations \eqref{eq:L1_4_quantum_limit_defn}, \eqref{eq:L5_8_quantum_limit_defn}, and \eqref{eq:L9_12_quantum_limit_defn}.

To prove these results, we relate the functions $L_j$ to Maass waveforms with the technology of mock Maass theta functions developed by Zwegers in \cite{Zw}. These objects give a rare and precious glimpse into the behavior of false-indefinite theta functions under modular transformations. As reviewed in Figure \ref{fig:mock_maass_diagram}, mock Maass theta functions are certain theta functions that are eigenfunctions of the hyperbolic Laplacian that are in general not modular. Their construction ensures that they give rise to false-indefinite theta functions under certain ``Eichler-type integrals'' following the work of Lewis and Zagier \cite{LZ}. The one-form appearing in the integral is closed thanks to the hyperbolic Laplacian eigenfunction property of the mock Maass theta function. If it is further true that the mock Maass theta function is modular invariant, then we can find the obstruction to modularity for the corresponding false-indefinite theta function as a period function of the aforementioned one-form. The question of modularity, on the other hand, can be studied through Zwegers' modular completions for these mock Maass theta functions. In the special cases where the completing ``shadow contributions" cancel each other thanks to symmetry, the mock Maass theta function becomes a Maass waveform itself and the modular properties of the false-indefinite theta function then follows.

\begin{figure}[h!]
 \vspace{-8pt}
  \centering
    \includegraphics[scale=0.27]{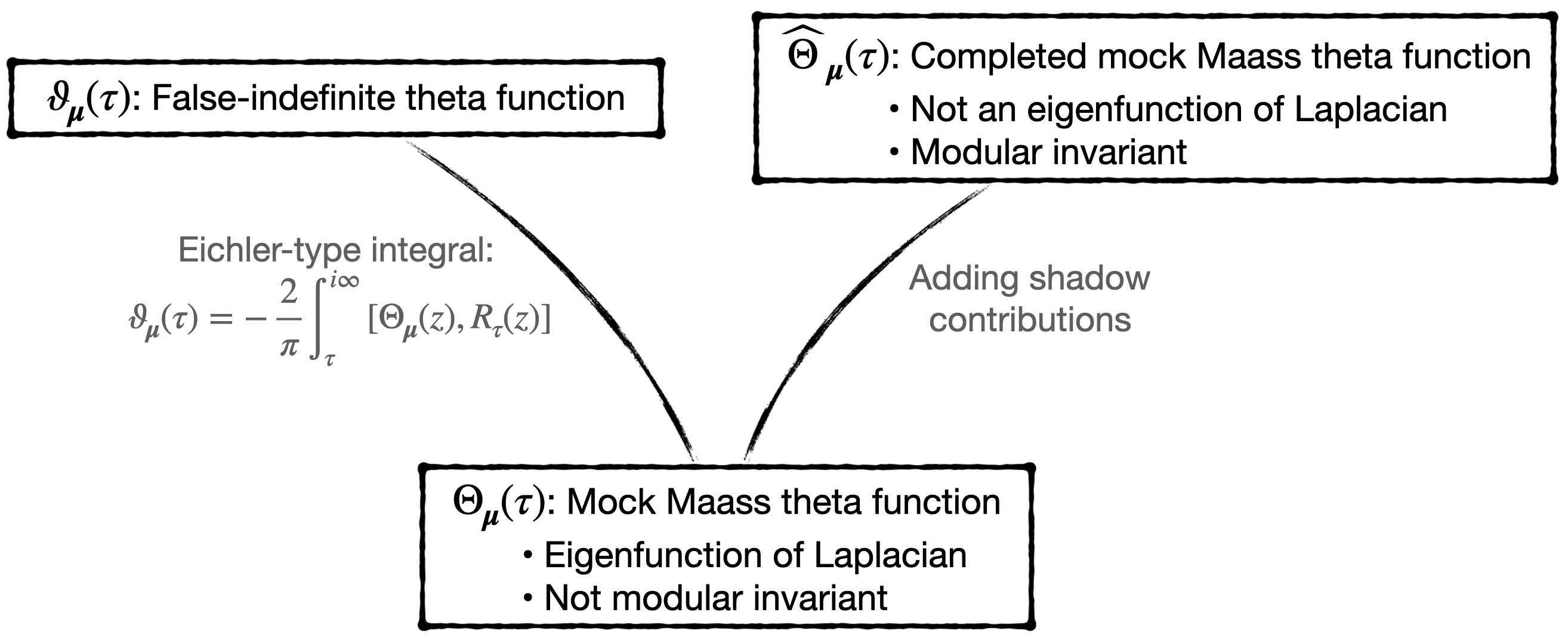}
     \vspace{-10pt}
    \caption{}
    \label{fig:mock_maass_diagram}
\end{figure}

If the functions $L_j$ are rewritten as false-indefinite theta functions, such cancellations indeed occur as we show below and the corresponding mock Maass theta functions are modular.
That in turn implies corresponding modular properties for the functions $L_j$ on $\H$ as discussed above. For the case of $L_1,\dots,L_4$ and $L_9,\dots,L_{12}$, these Maass forms do not have constant terms in their Fourier expansion and the arguments of Zagier \cite{Za} immediately apply to give the quantum modularity results stated above. However, for the functions $L_5,\dots,L_8$, the corresponding vector-valued Maass form does have a nonvanishing constant term and as a result the corresponding $L_j(e^{2\pi i(x+it)})$'s have divergent pieces as $t \to 0^+$ depending on $j$ and $x$. So a novel technical aspect of this work is the handling of these divergent pieces to show that the remaining finite pieces lead to quantum modular forms.

This paper is organized as follows. In Section \ref{sec:maass_forms}, we review and expand on the relation between Maass forms and quantum modular forms, with a particular emphasis on the aspects needed if the Maass form in question has nonvanishing constant terms. In Section \ref{sec:mock_maass_theta_functions}, we review the work of Zwegers on mock Maass theta functions and their relation to false-indefinite theta functions that appear in this work. Then in Sections \ref{sec:L1_L4_analysis}, \ref{sec:L5_L8_analysis}, \ref{sec:L9_L12_analysis}, we analyze the functions $L_1,\dots,L_4$, $L_5,\dots,L_8$, and $L_9,\dots,L_{12}$, respectively, and prove the results asserted in Theorem \ref{thm:main_result}. Finally, in an appendix we display various numerical results that exemplify our discussion in the body of this paper.

\section*{Acknowledgements}
The authors were supported by the SFB/TRR 191 “Symplectic Structure in Geometry, Algebra and Dynamics”, funded by the DFG (Projektnummer 281071066 TRR 191)
and by the Deutsche Forschungsgemeinschaft (DFG) Grant No. BR 4082/5-1.

\section{Maass Forms and Quantum Modular Forms}\label{sec:maass_forms}
In this section we consider quantum modular forms related to Maass forms. First we recall the definition of a (vector-valued) Maass form.
\begin{defn}\label{defn:Maass_form}
A set of smooth functions $U_j: \H \to \C$ with $j \in\{1,\dots,N\} $ is called \textit{a vector-valued Maass (wave) form} for the group $\GG\subset\SL_2(\Z)$ if it satisfies the following conditions:
\begin{enumerate}[leftmargin=15pt,label=\arabic*)]
\item For all $M=\pmat{a&b\\b&c} \in \GG$ we have 
\begin{equation}\label{eq:U_modular_transformation}
U_j \left(\frac{a\t+b}{c\t+d}\right) = \sum_{k=1}^N \Psi_M(j,k)  U_k(\t),
\end{equation}
where $\Psi_M$ is a suitable multiplier system.\footnote{Here and throughout we assume that $\Psi_M$ is diagonal on $\GG_{\infty}:= \GG \cap \{ \pm \pmat{1&n\\0&1}: \, n\in \Z\}$.}

\item There exists a $\l \in \C$ such that $\Delta \left(U_j\right) = \l U_j$ for each $j$, where $\Delta:=-\t_2^2(\frac{\del^2}{\del\t_1^2}+\frac{\del^2}{\del\t_2^2})$ is the hyperbolic Laplace operator on $\H$.

\item The functions $U_j$ have at most polynomial growth near the cusps.
\end{enumerate}
\end{defn}
The Maass forms $U_j$ with Laplace eigenvalue $\l = \frac{1}{4} - \nu^2$ (where $\nu \in \C$) have a Fourier expansion of the form
\begin{equation*}
U_j (\t) = \sum_{n\in\Z+\a_j} a_j(\t_2;n)  e^{2\pi in\t_1},
\end{equation*}
where, for some constants $d_j(n),b_j,c_j$,
\begin{equation*}
a_j(\t_2;n) = d_j(n)\sqrt{\t_2}K_\nu(2\pi|n|\t_2) \text{ if } n\ne0 \andd a_j(\t_2;0) =
\begin{cases}
b_j\log(\t_2)\sqrt{\t_2}+c_j\sqrt{\t_2} &\text{if }\nu=0,\\
b_j\t_2^{\frac12-\nu}+c_j\t_2^{\frac12+\nu} &\text{if }\nu\ne0.\\
\end{cases}
\end{equation*}

In this paper, our interest is on Maass forms with Laplace eigenvalue $\frac{1}{4}$. More specifically, we study functions $U_j:\H\to\C$, $j\in\{1,\dots,N\}$ forming a vector-valued Maass form for $\GG\subset\SL_2(\Z)$ with Fourier expansion
\begin{equation}\label{eq:U_fourier_expansion}
U_j(\t) = c_j \sqrt{\t_2} + \sqrt{\t_2} \sum_{\substack{n\in\Z+\a_j\\n\ne0}} d_j(n)K_0(2\pi|n|\t_2) e^{2\pi in\t_1},
\end{equation}
where the coefficients $d_j(n)$ have polynomial growth in $n$. Following \cite{LZ} we then define\footnote{Throughout this paper we use the principal branch of the logarithm to define the square-roots.}
\begin{equation*}
R_\t(z) := \frac{\sqrt{z_2}}{\sqrt{(z-\t)\left(\bar{z}-\t\right)}}
\andd
[U_j(z),R_\t(z)] := \del U_j(z) R_\t(z)  dz + U_j(z) \bar{\del} R_\t(z) d\bar{z},
\end{equation*}
where $\del f(z) := \frac{\del f(z)}{\del z}$ and $\bar{\del} f(z) := \frac{\del f(z)}{\del \bar{z}}$. For each $j$, the one-form $[U_j(z),R_\t(z)]$ is closed thanks to the fact that both $R_\t$ and $U_j$ have eigenvalue $\frac{1}{4}$ under the Laplacian. Using these closed one-forms, we define $u_j:\H\to\C$ and $\mathcal{U}_{j,\varrho}:\C \setminus \lp \varrho + i \R \rp \to\C$ by\footnote{We also define $\mathcal{U}_{j,i\infty}(\t):=0$.}${}^{,}$\footnote{
The latter function is a period function that can be analytically continued in $\tau$ to a cut complex plane and that satisfies a functional equation under modular transformations. The correspondence between Maass cusp forms and their period functions was first elucidated by Lewis and Zagier in \cite{LZ}, where they also recognized its role similar to that of period polynomials for holomorphic cusp forms.
}
\begin{equation}\label{eq:defn_u_and_error}
u_j(\t) := -\frac2\pi\int_\t^{i\infty} [U_j(z),R_\t(z)]
\andd
\mathcal{U}_{j,\varrho}(\t) := \frac2\pi\int_{\varrho}^{i\infty} [U_j(z),R_\t(z)],
\end{equation}
for any
\begin{equation}\label{eq:Qgammma_definition}
\varrho \in \mathcal{Q}_\GG := \left\{ x \in \Q : \ \mbox{there exists } M_x := \mat{a_x&b_x\\c_x&d_x}\in \GG \mbox{ with } x = -\frac{d_x}{c_x} \right\} .
\end{equation}
The integrals are independent of the integration path, but for concreteness we assume throughout that the integrals are taken along vertical paths. Now the results of Lewis and Zagier in \cite{LZ} and of Zagier in \cite{Za} show that if the $U_j$'s form a Maass form with $c_j = 0$ for all $j$, then the limits $\lim_{t \to 0^+} u_j (x+it)$ exist for all $x \in \mathcal{Q}_\GG$ and they form a quantum modular form for $\GG$.

Among the functions we study, $L_1,\dots,L_4$ and $L_9,\dots,L_{12}$ indeed fit into this framework. The functions $L_5,\dots,L_8$, on the other hand, are related to a vector-valued Maass form with a nontrivial constant term. So in this section we develop the technical details for how the results of \cite{LZ} and \cite{Za} extend to this case.\footnote{In \cite{LZ}, a discussion of noncuspidal $\SL_2(\Z)$ Maass forms was given (see equations (4.4) and (4.5) there with $s=\frac12-\nu$). However, note that the constant term $b_j\t_2^{\frac12-\nu}+c_j\t_2^{\frac12+\nu}$ degenerates for $\nu=0$, which is the case of interest in this paper, and we have not found an interpretation for the constant term given in (4.5) of \cite{LZ} (that would correspond to the constant term of $u_j$) to reproduce our result in Proposition \ref{prop:uj_fourier_expansion}.}

We start with an elementary result that relates the Maass waveforms discussed here to $q$-series.
\begin{prop}\label{prop:uj_fourier_expansion}
The functions $u_j$ for $j\in\{1,\dots,N\}$ are holomorphic on $\H$ and they satisfy
\begin{equation*}
u_j(\t) = -\frac{c_j}{\pi} + \sum_{\substack{n\in\Z+\a_j\\n>0}} d_j(n)  q^n.
\end{equation*}
\end{prop}
\begin{proof}
This basically follows from the results of \cite{LZ}. A detailed exposition for the case $c_j = 0$ can be found in Proposition 3.5 of \cite{LNR}. The contribution of the constant terms $c_j$ follow from a straightforward computation.
\end{proof}

Next we note the analytic properties of the functions $\mathcal{U}_{j,\varrho}$.
\begin{prop}\label{prop:modularity_error_holomorphy}
For $j \in  \{1,\dots,N\}$ and $\varrho \in  \mathcal{Q}_\GG$, the integral defining $\mathcal{U}_{j,\varrho}(z)$ is convergent for all $z \in \C \setminus (\varrho + i \R)$ and defines a holomorphic function there.
\end{prop}
\begin{proof}
Again this result basically follows from \cite{LZ}. Here we give details both for reference and also to point out the new ingredients that appear in the presence of constant Fourier coefficients. We start by writing 
\begin{equation*}
\calU_{j,\varrho}(\t) = \frac{1}{2\pi}\int_0^\infty \left(4it\del U_j(\varrho+it)-\frac{t+i(\t-\varrho)}{t-i(\t-\varrho)}U_j(\varrho+it)\right) \frac{dt}{\sqrt{t}\sqrt{t^2+(\t-\varrho)^2}}.
\end{equation*}
Then letting $\varrho=-\frac dc$ for $M = \pmat{a & b \\ c & d} \in \GG$, we can use the modular transformations \eqref{eq:U_modular_transformation} to write (for any $T >0$)
\begin{align}
\label{eq:error_modularity_two_integral}
&\mathcal{U}_{j,\varrho}(\t) = \frac{1}{2\pi} \int_{T}^{\infty} \lp
4it  \del U_j \lp -\frac{d}{c}+it \rp
-  \frac{t + i \lp \t+\frac{d}{c} \rp}{t - i \lp \t +\frac{d}{c} \rp}
  U_j \lp -\frac{d}{c} + i t \rp  
\rp \frac{dt}{\sqrt{t} \sqrt{ t^2 + \lp \t +\frac{d}{c} \rp^2 } }
\\
&
-\frac{1}{2\pi} \sum_{k=1}^N \Psi_{M^{-1}} (j,k)  \int_{0}^{T} \lp
\frac{4i}{c^2t} \del U_k \lp \frac{a}{c} + \frac{i}{c^2 t} \rp
+ \frac{t + i \lp \t+\frac{d}{c} \rp}{t - i \lp \t +\frac{d}{c} \rp}
 U_k \lp \frac{a}{c} + \frac{i}{c^2 t} \rp
\rp \frac{dt}{\sqrt{t} \sqrt{ t^2 + \lp \t +\frac{d}{c} \rp^2 } }.
\notag
\end{align}
Now looking at the Fourier expansion in
\eqref{eq:U_fourier_expansion}, we find that for any $\varepsilon > 0$ we have constants $C_\varepsilon, D > 0$ such that for all $\t_2 \geq \varepsilon$ and for all $j \in \{1,\ldots,N\}$ we have 
\begin{equation}\label{eq:Uj_delUj_bounds1}
\left| U_j (\t) - c_j \sqrt{\t_2} \right|, \quad 
\left| 4 i \del U_j (\t) - \frac{c_j}{\sqrt{\t}}  \right| \leq C_\varepsilon e^{-D \t_2} 
\end{equation}
thanks to the exponential decay of $K$-Bessel functions towards infinity. So separating the constant terms of $U$ and $\del U$'s, it is easy to see that their contribution to \eqref{eq:error_modularity_two_integral} is convergent and yields a holomorphic function. The contribution of the constant terms can then be separately checked to be convergent and to be holomorphic for all $z \in \C \setminus (\varrho + i \R)$.
\end{proof}

\begin{rem}\label{rem:error_modularity_holomorphic_extension_cut_plane}
Here it is also useful to note that
by deforming the path of integration and the cut associated with the square-roots as in Figure \ref{fig:obstruction_modularity_continuation},
one can analytically continue $\mathcal{U}_{j,\varrho}(z)$ from the half-plane $\mathrm{Re} (z) > \varrho$ to the entire cut-plane $\C \setminus (-\infty, \varrho]$.\footnote{Similarly we can continue $\mathcal{U}_{j,\varrho}(z)$ from the half-plane $\mathrm{Re} (z) < \varrho$ to the entire cut-plane $\C \setminus [\varrho, \infty)$.} 
\end{rem}
\begin{figure}[h!]
 \vspace{-10pt}
  \centering
    \includegraphics[scale=0.20]{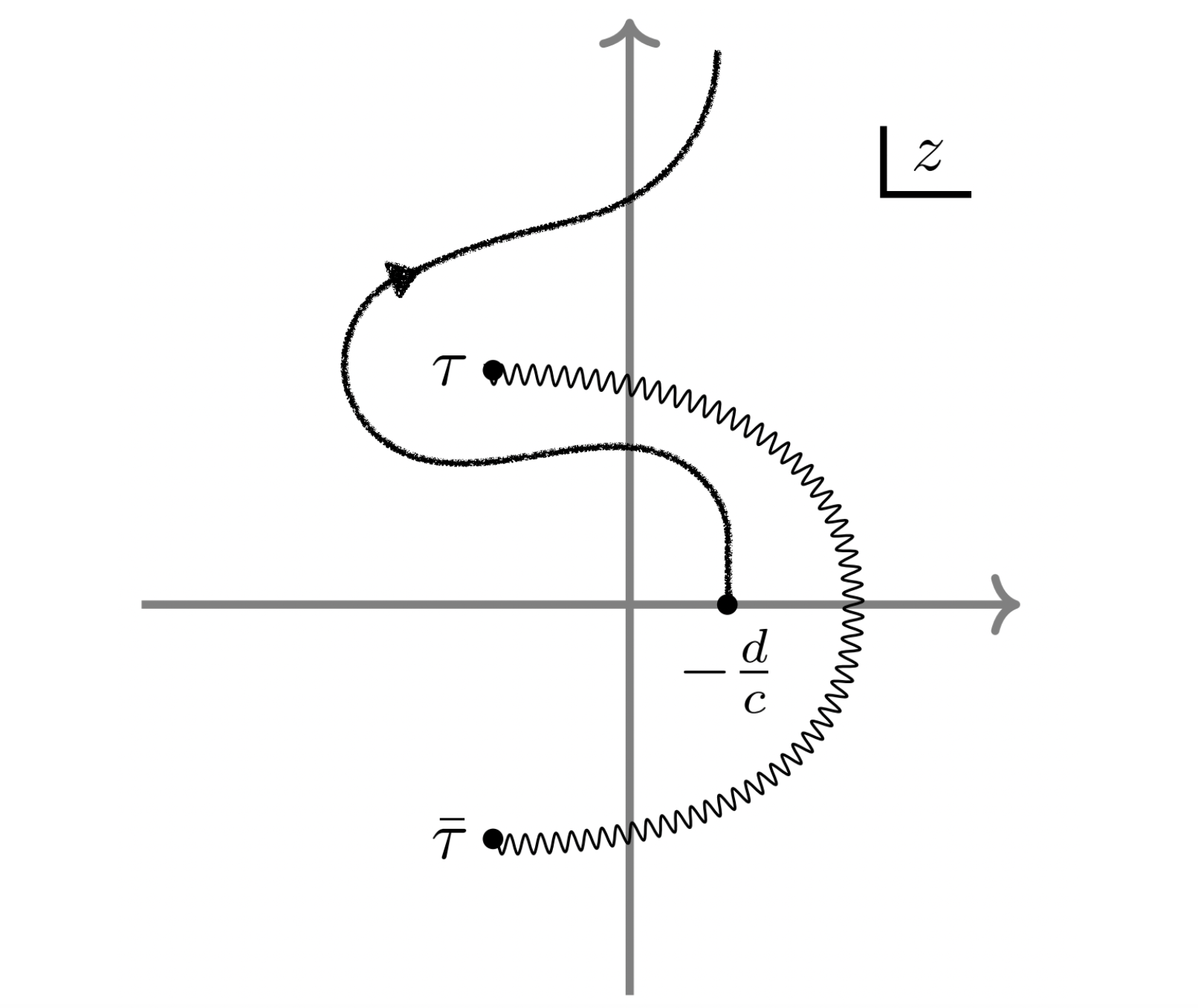}
     \vspace{-10pt}
    \caption{}
    \label{fig:obstruction_modularity_continuation}
\end{figure}

Now we are ready to state the modular transformations of $u_j$ on $\H$.
\begin{prop}\label{prop:uj_modularity}
For $M:=\pmat{a&b\\c&d}\in\GG$ and $\t\in\H$ with $\t_1\ne-\frac dc$ we have
\begin{equation*}
u_j\left(\frac{a\t+b}{c\t+d}\right) =  \sgn(c\t_1+d)(c\t+d) \sum_{k=1}^N \Psi_M(j,k)\left(u_k(\t)+\calU_{k,-\frac dc}(\t)\right).
\end{equation*}
\end{prop}

\begin{proof}
The statement is trivial for $c=0$, so we assume that $c\ne0$. We start with the definition in equation \eqref{eq:defn_u_and_error} for $u_j(\frac{a\t+b}{c\t+d})$ and make a change of variable $z\mapsto\frac{az+b}{cz+d}$ to write
\begin{equation*}
u_j\left(\frac{a\t+b}{c\t+d}\right) = -\frac2\pi\int_\t^{-\frac dc} \left[U_j\left(\frac{az+b}{cz+d}\right),R_{\frac{a\t+b}{c\t+d}}\left(\frac{az+b}{cz+d}\right)\right],
\end{equation*}
where the integral is taken over a piece of hyperbolic geodesic from $\t$ to $-\frac dc$. Now we note that
\begin{equation*}
R_\frac{a\t+b}{c\t+d}\left(\frac{az+b}{cz+d}\right) = \chi_M(\t,z)(c\t+d)R_\t(z),
\end{equation*}
where
\begin{equation*}
\chi_M(\t,z) := \sqrt{\frac{(c\t+d)^2}{(z-\t)\left(\ol z-\t\right)}} \frac{\sqrt{(z-\t)\left(\ol z-\t\right)}}{c\t+d} \in \{\pm1\}.
\end{equation*}
Since we are using the principal values of the logarithm to define square roots, $R_\t(z)$ is discontinuous on the $z$-plane along the vertical cut from $\t$ to $\ol\t$. Then following the modular transformation, the sign factor $\chi_M(\t,z)$ is discontinuous along this vertical line and along another piece of a hyperbolic geodesic emanating from $\t$ (to which the vertical cut of $R_\t(z)$ transforms). Moreover, on these discontinuity cuts themselves, the sign factor $\chi_M(\t,z)$ takes a constant value (since $R_\t(z)$ is continuous if it is restricted to its vertical cut). In particular, $\chi_M(\t,z)$ is constant along our integration line, which itself is a hyperbolic geodesic from $\t$ to $-\frac dc$. We can determine this sign to be $\sgn(c\t_1+d)$ by setting $z=-\frac dc$. Then using the modular transformation in \eqref{eq:U_modular_transformation} and using our freedom to deform the path of integration to write $\int_\t^{-\frac dc}=\int_\t^{i\infty}-\int_{-\frac dc}^{i\infty}$ we obtain the result.
\end{proof}

Our next goal is to define a function on rationals by taking vertical limits of $u_j (\t)$ and removing the growing pieces.

\begin{prop}\label{prop:uj_vertical_limit}
For any $M:=\pmat{a&b\\c&d}\in\GG$ with $c\ne0$, the limit
\begin{equation}\label{eq:uj_vertical_limit}
\lim_{t\to0^+} \left(u_j\left(-\frac dc+it\right)-\frac{1}{\pi|c|t} \sum_{k=1}^N \Psi_{M^{-1}}(j,k)  c_k\right)
\end{equation}
exists. Moreover, the same limit is obtained for any element of $\GG$ in the same equivalence class as $M$ in $\GG_\infty\backslash\GG$. Using this limit we can define the functions $\mathfrak{u}_j:\calQ_\GG\to\C$ by
\begin{equation*}\label{eq:frak_uj_definition}
\mathfrak{u}_j(x) := \lim_{t\to0^+}\left(u_j(x+it)-\frac{\g_{j,x}}{\pi t}\right) \quad \mbox{with } \g_{j, x} := \frac{1}{|c_x|} \sum_{k=1}^N \Psi_{M_x^{-1}}(j,k)  c_k.
\end{equation*}
Here $x=-\frac{d_x}{c_x}$ and $M_x=\pmat{a_x&b_x\\c_x&d_x}\in\GG$ as in the definition of $\calQ_\GG$ in \eqref{eq:Qgammma_definition}.
\end{prop}

\begin{proof}
We start by explicitly writing the expression for $u_j (\t)$ in \eqref{eq:defn_u_and_error} as 
\begin{equation}\label{eq:uj_explicit_version}
u_j (\t) = \frac{1}{2 \pi} 
\int_0^\infty \lp - 4 i  \del  U_j(\t+iv) \frac{\sqrt{v+\t_2}}{\sqrt{v} \sqrt{v+2\t_2}}
+
U_j(\t+iv) \frac{\sqrt{v}}{\sqrt{v+\t_2} (v+2\t_2)^{\frac32}}
\rp  dv .
\end{equation}
Inserting $\t = -\frac{d}{c} + it$ and separating the integral into two pieces as $\int_0^1$ and $\int_1^\infty$, we easily find that the limit $t \to 0^+$ exists for the latter piece and can be computed by setting $t=0$ thanks to the bounds in \eqref{eq:Uj_delUj_bounds1}. For the integral $\int_0^1$, on the other hand, we use the modular transformations in \eqref{eq:U_modular_transformation} and separate the constant term to rewrite this contribution as
\begin{align*}
&- \frac{1}{2 \pi c^2} \sum_{k=1}^N \Psi_{M^{-1}} (j,k)
\int_0^1 \lp - 4 i \del  U_k \lp  \frac{a}{c}+ \frac{i}{c^2(v+t)} \rp 
+ c_k |c| \sqrt{v+t} \rp
\frac{1}{\sqrt{v} \sqrt{v+2t} (v+t)^{\frac32}}  dv
\\ & \qquad 
+\frac{1}{2 \pi} \sum_{k=1}^N \Psi_{M^{-1}} (j,k)
\int_0^1 \lp U_k \lp  \frac{a}{c}+ \frac{i}{c^2(v+t)} \rp - \frac{c_k}{|c|\sqrt{v+t}} \rp
\frac{\sqrt{v}}{\sqrt{v+t} (v+2t)^{\frac32}}  dv
\\ & \qquad \qquad 
+\frac{1}{\pi |c|} \sum_{k=1}^N \Psi_{M^{-1}} (j,k)  c_k 
\int_0^1  \frac{dv}{\sqrt{v} (v+2t)^{\frac32}} .
\end{align*}
Next we use a looser version of the bound in \eqref{eq:Uj_delUj_bounds1}, namely that for any $\varepsilon > 0$ we have a constant $B_\varepsilon > 0$ such that for all $\t_2 \geq \varepsilon$ and for all $j \in \{1,\ldots,N\}$ we have 
\begin{equation}\label{eq:Uj_delUj_bounds2}
\left| U_j (\t) - c_j \sqrt{\t_2} \right|, \quad
\left| 4 i \del U_j (\t) - \frac{c_j}{\sqrt{\t_2}} \right| \leq \frac{B_\varepsilon}{\t_2^2}.
\end{equation}
Using these two bounds, we find that for $t, v \leq 1$ the integrands in the first and second lines can be bounded by $\frac{C}{\sqrt{v}}$ and $C \sqrt{v}$, respectively, for an appropriate constant $C>0$. Thanks to these upper bounds and to the fact that the integrands are continuous at $t =0$, for those two terms, the limit $t \to 0^+$ exists and they can simply be taken by explicitly setting $t=0$. The third line, on the other hand, can be explicitly evaluated and we find that the piece of it that grows as $t \to 0^+$ is removed by the term $\frac{1}{\pi |c| t} \sum_{k=1}^N \Psi_{M^{-1}} (j,k)  c_k$ in \eqref{eq:uj_vertical_limit}.

Finally, the fact that the same limit is obtained for any $\pmat{a + rc & b + rd \\ c & d} \in \GG$ (with $r \in \Z$) follows from the diagonality of the multiplier system over $\GG_\infty$.
\end{proof}

Now the functions $\mathfrak{u}_j$ are defined through vertical limits in $\H$, we show how the same functions are obtained when we slightly deform the path through which we take the limit.

\begin{lem}\label{lem:frak_uj_limit_path_deformation}
For any $x \in \mathcal{Q}_\GG$ 
and for any smooth function $X: \R^+ \to \R$ that satisfies $X (t) = B t^2 + o(t^2)$ as $t \to 0^+$ we have
\begin{equation*}
\mathfrak{u}_j (x) = \lim_{t \to 0^+} \lp
u_j \lp x + i t + X (t) \rp -  \frac{\g_{j, x}}{\pi (t-i X(t))} \rp .
\end{equation*}
\end{lem}

\begin{proof} 
We prove the equivalent statement
\begin{equation}\label{eq:uj_rational_limit_deformed}
\lim_{t \to 0^+} \lp
u_j \lp x + i t + X (t) \rp 
- u_j \lp x + i t \rp \rp =  \frac{i B  \g_{j, x}}{\pi} .
\end{equation}
We start by using equation \eqref{eq:uj_explicit_version} to write
\begin{align*}
&u_j(x+it+X(t)) - u_j(x+it)
\\
& \qquad \qquad
= \frac{1}{2\pi} \int_0^\infty (-4i\del U_j(x+X(t)+i(v+t))+4i\del U_j(x+i(v+t))) \frac{\sqrt{v+t}}{\sqrt{v}\sqrt{v+2t}} dv
\\ & \qquad \qquad\qquad\qquad
+ \frac{1}{2 \pi} \int_0^\infty (U_j(x+X(t)+i(v+t))-U_j(x+i(v+t))) \frac{\sqrt{v}}{\sqrt{v+t}(v+2t)^{\frac32}} dv.
\end{align*}
As in Proposition \ref{prop:uj_vertical_limit}, we separate the integrals into two pieces as $\int_0^1$ and $\int_1^\infty$. For the contributions from $\int_1^\infty$, the same argument that we use there can be employed to show that the limit $t\to0^+$ can be computed by explicitly setting $t=0$ inside the integrands. In particular, this shows that the contributions from $\int_1^\infty$ do not contribute to the limit in \eqref{eq:uj_rational_limit_deformed}.

To study the contributions from $\int_0^1$, we let $M=\pmat{a&b\\c&d}\in\GG$ be such that 
$x=-\frac{d}{c}$ and use modular transformations by $M^{-1}$ to rewrite them as
\begin{align*}
&\frac{1}{2\pi c^2} \sum_{k=1}^N \Psi_{M^{-1}}(j,k) \int_0^1 4i\del U_k\left(\frac{a}{c}+\frac{i}{c^2(v+t-iX(t))}\right) \frac{\sqrt{v+t}}{\sqrt{v}\sqrt{v+2t}(v+t-iX(t))^2} dv\\
&\qquad - \frac{1}{2\pi c^2} \sum_{k=1}^N \Psi_{M^{-1}}(j,k) \int_0^1 4i\del U_k\left(\frac{a}{c}+\frac{i}{c^2(v+t)}\right) \frac{1}{\sqrt{v}\sqrt{v+2t}(v+t)^\frac32} dv\\
&\qquad\qquad + \frac{1}{2\pi} \sum_{k=1}^N \Psi_{M^{-1}}(j,k) \int_0^1 U_k\left(\frac{a}{c}+\frac{i}{c^2(v+t-iX(t))}\right) \frac{\sqrt{v}}{\sqrt{v+t}(v+2t)^{\frac32}} dv\\
&\qquad\qquad\qquad - \frac{1}{2\pi} \sum_{k=1}^N \Psi_{M^{-1}}(j,k) \int_0^1 U_k\left(\frac{a}{c}+\frac{i}{c^2(v+t)}\right) \frac{\sqrt{v}}{\sqrt{v+t}(v+2t)^{\frac32}} dv.
\end{align*}
Now we choose $T \leq 1$ to be small enough that $| X(t)| \leq t$ for all $t \in [0,T]$. Then for $v \in (0,1)$ 
\begin{equation*}
\mathrm{Im} \lp  \frac{a}{c} + \frac{i}{c^2 \lp v+t - i X(t) \rp} \rp \geq \frac{1}{4c^2} .
\end{equation*}
As in the proof of Proposition \ref{prop:uj_vertical_limit}, this allows us to use the bounds in \eqref{eq:Uj_delUj_bounds2} to show that once we remove the constant terms of $U_k$ and $\del U_k$'s we can take $t\to0^+$ by setting $t=0$ inside the integrands and hence we get no contribution to the limit in \eqref{eq:uj_rational_limit_deformed} from such terms.
This leaves us the contribution from the constant terms:
\begin{align*}
\frac{1}{2 \pi |c|} &\sum_{k=1}^N \Psi_{M^{-1}} (j,k)  c_k  
\int_0^1 \lp
\frac{\sqrt{X(t)^2 + (v+t)^2}}{\sqrt{v} \sqrt{v+2t} \lp v+ t - i X(t) \rp^2}
-
\frac{1}{\sqrt{v} \sqrt{v+2t} (v+t)} \right.
\\  & \hspace{65mm}
\left.
+ \frac{\sqrt{v}}{\sqrt{X(t)^2 + (v+t)^2} (v+2t)^{\frac32}}
- \frac{\sqrt{v}}{(v+t) (v+2t)^{\frac32}}
\rp dv \\
&=\frac{\g_{j, x}}{\pi t} 
 \int_0^{\frac1t}
\frac{1}{\sqrt{v} (v+2)^{\frac32}} 
\left(
\left(
{ \frac{(v+1)^3}{\left( \frac{X(t)^2}{t^2} + (v+1)^2 \right)^{\frac32}} -1 }
\right)
\left( 1 + \frac{i X(t)}{t} \frac{v+2}{(v+1)^2} \right) 
 \right.
\\
& \left.\hspace{30mm}
\phantom{\frac{(v+1)^3}{\left( \frac{X(t)^2}{t^2} + (v+1)^2 \right)^{\frac32}}}
- \frac{X(t)^2}{t^2} \frac{1}{\left( \frac{X(t)^2}{t^2} + (v+1)^2 \right)^{\frac32}} 
+\frac{i X(t)}{t}\frac{v+2}{(v+1)^2}
\right)  dv, 
\end{align*}
from which it is easy to see that the first two terms in the outer parentheses give vanishing contributions as $t \to 0^+$, whereas the contribution of $\frac{i X(t)}{t}\frac{v+2}{(v+1)^2}$ gives the right-hand side of \eqref{eq:uj_rational_limit_deformed} and thereby concludes our proof.
\end{proof}

We are now ready to state the quantum modular transformation properties of the functions $\mathfrak{u}_j$.
\begin{thm}\label{thm:frak_uj_quantum_modular_transformations}
For any $M := \pmat{a&b\\c&d}\in \GG$ and $x \in \mathcal{Q}_\GG \setminus \{-\frac{d}{c} \}$  we have the quantum modular transformation property
\begin{equation*}
\mathfrak{u}_j\lp \frac{ax+b}{cx+d}\rp 
= |cx+d| \sum_{k=1}^N \Psi_M(j,k) 
\lp \mathfrak{u}_k(x)+\mathcal{U}_{k,-\frac{d}{c}}(x)\rp.
\end{equation*}
\end{thm}
\begin{proof}
We start by noting that $\Psi_M$ is a multiplier system for a weight zero modular object (see equation \eqref{eq:U_modular_transformation}). Therefore, it forms a genuine representation of the modular group $\GG$ (as opposed to a projective one) and this leads to the identity
\begin{equation}\label{eq:gamma_modular_property}
\sum_{k=1}^N \Psi_M(j,k)   \g_{k,x} 
=
|cx+d|   \g_{j, \frac{ax+b}{cx+d}} . 
\end{equation}
So plugging in $\t = x+it$ (with $t>0$) to the modular transformation worked out in Proposition \ref{prop:uj_modularity} and using \eqref{eq:gamma_modular_property} we obtain
\begin{align*}
&u_j \lp \frac{ax+b}{cx+d} + \frac{it}{(cx+d)(cx+cit+d)} \rp
- \frac{\g_{j,\frac{ax+b}{cx+d}}}{\pi t}(cx+d)(cx+cit+d)
\\ & \qquad \qquad 
=
\sgn(cx+d)(cx+cit+d) \sum_{k=1}^N \Psi_M(j,k) 
\lp u_k(x+it) - \frac{\g_{k,x}}{\pi t}+\mathcal{U}_{k,-\frac{d}{c}}(x+it)\rp .
\end{align*}
By Propositions \ref{prop:modularity_error_holomorphy} and \ref{prop:uj_vertical_limit}, the right-hand side tends to the right-hand side of the claim as $t \to 0^+$. Finally, we note that the left-hand side tends to $\mathfrak{u}_j( \frac{ax+b}{cx+d})$ as $t \to 0^+$ by Lemma \ref{lem:frak_uj_limit_path_deformation}.
\end{proof}

\section{Mock Maass Theta functions}\label{sec:mock_maass_theta_functions}

In the following sections, we see that the functions $L_1,\dots,L_{12}$ can all be expressed in terms of certain theta functions of the form\footnote{Note that throughout we write vectors in bold letters.}
\begin{equation}\label{eq:indefinite_false_theta}
\sum_{\bm n\in\Z^d+\bm\mu} (1-\sgn(B(\bm n,\bm{c_1}))\sgn(B(\bm n,\bm{c_2})))q^{Q(\bm n)},
\end{equation}
where $Q$ is a quadratic form of indefinite signature and the vectors $\bm{c_1}$ and $\bm{c_2}$ of negative norm ensure convergence. Such functions are not as well-studied as similar looking indefinite theta functions of the shape 
\begin{equation*}
\sum_{\bm n\in\Z^d+\bm\mu} (\sgn(B(\bm n,\bm{c_1}))-\sgn(B(\bm n,\bm{c_2})))q^{Q(\bm n)},
\end{equation*}
which are known to yield mock modular forms thanks to the ground-breaking work of Zwegers \cite{Zw2}. The functions in equation \eqref{eq:indefinite_false_theta}, which are hybrids of indefinite and false theta functions, do not easily fit into a modular framework; but they are still interestingly related to the so-called mock Maass theta functions developed by Zwegers in \cite{Zw}. In this section, we review the properties of such functions for our later use.

First, we restrict ourselves to $2$-dimensional lattices and let $Q$ be a binary quadratic form of signature $(1,1)$. Throughout the paper we assume that $Q(\bm n)$ is integral for all $\bm n=(n_1\ n_2)^T\in\Z^2$ (in which case we call the quadratic form {\it even}) and that $Q(\bm n)=\frac12\bm n^TA\bm n$, where $A$ is a  symmetric $2\times2$ matrix. Finally, we let $B$ denote the bilinear form associated to $Q$ as $B(\bm n,\bm m)=\bm n^TA\bm m$.

Next, we recall that the set of vectors $\bm c\in\R^2$ with $Q(\bm c)=-1$ splits into two connected components. Fixing a vector $\bm{c_0}$ in one of the components, all the vectors in the same component are characterized by 
\begin{equation*}
C_Q:=\left\{ \bm{c}\in\R^2 : Q(\bm c)=-1, B(\bm{c},\bm{c_0})<0 \right\}.
\end{equation*}
In particular, if $B(\bm{c_1} ,\bm{c_2})<0$, then $\bm{c_1}$ and $\bm{c_2}$ belong to the same component.  
Similarly, the set of vectors $\bm c \in \R^2$ with $Q(\bm c)=1$ splits into  two components, as well. Choosing $\bm{c_0^\perp}$ as one of the two unit vectors that are orthogonal to $\bm{c_0}$, all the unit vectors in the same component are given by 
\begin{equation*}
C_Q^\perp := \left\{ \bm{c}\in\R^2 : Q(\bm c)=1, B\left(\bm{c},\bm{c_0^\perp}\right) > 0 \right\}.
\end{equation*}

It is convenient to parametrize $C_Q$ and $C_Q^\perp$ by using the reference quadratic form $Q_0(\bm x)=x_1^2-x_2^2$. For this, we let $P\in\GL_2(\R)$ be such that
\begin{equation*}
A=P^T \mat{2 & 0 \\ 0 & -2} P,
\qquad 
P^{-1} \mat{0 \\ 1}  \in C_Q,
\andd
P^{-1} \mat{1 \\ 0}  \in C_Q^\perp .
\end{equation*}
Then we parametrize the vectors in $C_Q$ and $C_Q^\perp$ by letting
\begin{equation}\label{eq:ct_definition}
\bm{c(t)}:= P^{-1}  \mat{\sinh(t) \\ \cosh(t)} \in C_Q
\andd
\bm{c^\perp(t)}:= P^{-1}  \mat{\cosh(t) \\ \sinh(t)} \in C_Q^\perp 
\quad
\mbox{for } t \in \mathbb{R} .
\end{equation}
Also, when we consider a number of vectors $\bm{c_j} \in C_Q$ below, it is convenient to let  $t_j \in \R$ be such that $\bm{c(t_j)}=\bm{c_j}$ and let $\bm{c_j^\perp}=\bm{c^\perp(t_j)}$.

\begin{rem}\label{rem:P_choice}
The reference quadratic form we choose here is convenient since the quadratic forms in all of our examples are of the form $Q(\bm n)=\a_1n_1^2-\a_2n_2^2$ with $\a_1,\a_2\in\N$. We then choose $C_Q$ as the set of $\bm c\in\R^2$ with $Q(\bm c)=-1$ and $c_2>0$, whereas $C_Q^\perp$ as the set of $\bm c\in\R^2$ with $Q(\bm c)=1$ and $c_1>0$. Finally, note that we fix $P$ in all such examples by selecting $P=\pmat{\sqrt{\a_1}&0\\0&\sqrt{\a_2}}$ to satisfy the conditions in \eqref{eq:ct_definition}.\footnote{In \cite{Zw}, $Q_0(\bm x)=x_1x_2$ was used as the reference quadratic form and one picks a matrix $\calP\in\GL_2(\R)$ such that $A=\calP^T\pmat{0&1\\1&0}\calP$ and $\calP^{-1}\pmat{1\\-1}\in C_Q$ to parametrize $C_Q$. By letting $\calP=\pmat{1&1\\1&-1}P$, we can see the equivalence of the two definitions up to the extra requirement $P^{-1}\pmat{1\\0}\in C_Q^\perp$ here. To see why this extra condition is included here, note that one can multiply any $\calP$ as defined in \cite{Zw} on the left by $\pmat{0&-1\\-1&0}$. This preserves both of the conditions $A=\calP^T\pmat{0&1\\1&0}\calP$ and $\calP^{-1}\pmat{1\\-1}\in C_Q$. However, this changes the parameter $t$ on $C_Q$ as $t\mapsto-t$ (and in particular $t_j\mapsto-t_j$) while also changing $\bm{c_j^\perp}\mapsto-\bm{c_j^\perp}$. Consequently, this transformation changes the overall sign of the mock Maass theta functions defined in \cite{Zw}, whereas the extra condition here eliminates this ambiguity.}
\end{rem}

With this background at hand, we follow the Definition 2.3 of \cite{Zw} and define the completed mock Maass theta function (with $\bm\mu\in\R^2$ and $\bm{c_1}, \bm{c_2} \in C_Q$) 
\begin{equation}\label{eq:mock_maass_theta_fnc_completion}
\wh\mT_{\bm\mu}(\t) = \wh\mT_{\bm\mu}^{[\bm{c_1},\bm{c_2}]}(\t) := \sqrt{\t_2}\sum_{\bm n\in\Z^2+\bm\mu} q^{Q(\bm n)} \int_{t_1}^{t_2} e^{-\pi B(\bm n,\bm{c(t)})^2\t_2} dt.
\end{equation}
Note that it satisfies the basic properties
\begin{equation*}
\wh\mT_{-\bm\mu}(\t) = \wh\mT_{\bm\mu}(\t) \quad\mbox{and}\quad \wh\mT_{\bm\mu+\bm\l}(\t) = \wh\mT_{\bm\mu}(\t) \quad\mbox{for all } \bm\l \in \Z^2.
\end{equation*}
More importantly, such functions are covariant under modular transformations as shown in \cite{Zw}.
\begin{thm}[Zwegers]\label{thm:F_hat_modular_transformations}
Let  $Q(\bm n)=\frac12 \bm n^T A\bm n$ be an even quadratic form of signature $(1,1)$ on $\mathbb{Z}^2$ and let $\bm{\mu} \in A^{-1} \Z^2 / \Z^2$. Then we have the following transformations:
\begin{equation*}
\wh{\mT}_{\bm\mu}(\t+1) = e^{2\pi iQ(\bm\mu)}  \wh{\mT}_{\bm\mu}(\t),
\qquad 
\wh{\mT}_{\bm\mu}\left(-\frac1\t\right) = \frac{1}{\sqrt{|\det(A)|}} \sum_{\bm\nu\in A^{-1}\Z^2/\Z^2} e^{-2\pi iB(\bm\mu,\bm\nu)}  \wh{\mT}_{\bm\nu}(\t).
\end{equation*}
\end{thm}

The multipliers in Theorem \ref{thm:F_hat_modular_transformations} agree with the Weil representation associated to $Q$, so we can just as easily state the modular transformation under any element of $\SL_2(\Z)$ (see e.g.~ \cite{CS} for further details on Weil representations). We state this result more explicitly for later reference.
\begin{thm}\label{thm:That_general_modular_transformations}
Let  $Q(\bm n)=\frac12 \bm n^T A\bm n$ be an even quadratic form of signature $(1,1)$ on $\mathbb{Z}^2$ and let $\bm{\mu} \in A^{-1} \Z^2 / \Z^2$. Then, for any $M = \pmat{a&b\\c&d}\in\SL_2(\Z)$ we have
\begin{equation*}
\wh{\mT}_{\bm{\mu}}\lp \frac{a \tau + b}{c \tau + d} \rp 
=
\sum_{\bm{\nu} \in A^{-1} \Z^2 / \Z^2} 
\psi_M (\bm{\mu}, \bm{\nu})  \wh{\mT}_{\bm{\nu}} (\tau) ,
\end{equation*}
where 
\begin{equation*}
\psi_M(\bm\mu,\bm\nu) :=
\begin{cases}
e^{2 \pi i ab Q(\bm{\mu})} \delta_{\bm{\mu}, \sgn (d) \, \bm{\nu}} \quad &\mbox{if } c=0,\\[3pt]
\frac{1}{|c|  \sqrt{|\det(A)|}} 
\displaystyle\sum_{\bm{m} \in \Z^2 / c\Z^2}
e^{\frac{2\pi i}{c} \lp a Q(\bm{m}+\bm{\mu}) - B(\bm{m}+\bm{\mu}, \bm{\nu}) + d  Q(\bm{\nu}) \rp}
\quad &\mbox{if } c\neq 0.
\end{cases}
\end{equation*}
\end{thm}

The functions $\wh{\mT}_{\bm{\mu}}$ are closely related to the mock Maass theta functions defined by \cite{Zw}\footnote{
The exclusion of $\bm{n} = \bm{0}$ from the sums and the addition of the final term $(t_2-t_1) \sqrt{\t_2}  \d_{\bm\mu\in\Z^2}$ differs from \cite{Zw}, but this is an appropriate definition in view of Proposition \ref{prop:quadzero} below.
}${}^{,}$\footnote{
Here and throughout the paper $\sgn(x):=\frac{x}{|x|}$ for $x\in\R\setminus\{0\}$ and $\sgn(0):=0$.
}
\begin{align}
\mT_{\bm\mu}(\t) :=&
\frac12\sgn(t_2-t_1)\sqrt{\t_2} \sum_{\substack{\bm n\in \Z^2+\bm\mu \\ \bm{n} \neq \bm{0} }}   
\big( 1-\sgn(B(\bm n,\bm{c_1}))\sgn (B  (\bm n,\bm{c_2}  ) ) K_0(2\pi Q(\bm n)\t_2) e^{2\pi iQ(\bm n)\t_1}
\notag \\
+& \frac12\sgn(t_2-t_1)\sqrt{\t_2} \mspace{-8mu}
 \sum_{\substack{\bm n\in \Z^2+\bm\mu \\ \bm{n} \neq \bm{0} }}  \mspace{-12mu}
 \Big( 1-\sgn \Big(B  \Big(\bm n,\bm{c_1^\perp} \Big) \Big) \sgn\Big(B\big(\bm n,\bm{c_2^\perp}\Big)\Big) \Big) K_0(-2\pi Q(\bm n)\t_2) e^{2\pi iQ(\bm n)\t_1}
\notag \\
& \ + (t_2-t_1) \sqrt{\t_2} \d_{\bm\mu\in\Z^2} .
\label{eq:defn_mock_maass_theta_fnc}
\end{align}
These mock Maass theta functions are eigenfunctions of hyperbolic Laplacian with eigenvalue $\frac{1}{4}$ and hence are related to the theta functions we encounter in equation \eqref{eq:indefinite_false_theta} through equation \eqref{eq:defn_u_and_error} and Proposition \ref{prop:uj_fourier_expansion}. More specifically,
\begin{equation*}
\mt_{\bm\mu}(\t) := -\frac2\pi\int_\t^{i\infty} [\mT_{\bm\mu}(z),R_\t(z)] 
\end{equation*}
satisfies
\begin{align}
\mt_{\bm\mu}(\t) :=&
- \frac{t_2-t_1}{\pi} \d_{\bm\mu\in\Z^2}  +
\frac{\sgn(t_2-t_1)}{2} \sum_{\substack{\bm n\in \Z^2+\bm\mu \\ \bm{n} \neq \bm{0} }}   
\big( 1-\sgn(B(\bm n,\bm{c_1}))\sgn (B  (\bm n,\bm{c_2}  ) ) \big) q^{Q(\bm n)} .
\label{eq:defn_mock_maass_q_series}
\end{align}

Our first step towards understanding the relation between the mock Maass theta functions $\mT_{\bm\mu}$ and their completions $\wh{\mT}_{\bm{\mu}}$ is the following result (following from Lemma 4.1 of \cite{Zw}). 
\begin{prop}\label{prop:quadzero}
Let  $Q(\bm n)=\frac12 \bm n^T A\bm n$ be an even quadratic form of signature $(1,1)$ on $\mathbb{Z}^2$ and let $\bm{\mu} \in A^{-1} \Z^2 / \Z^2$. Also assume that $Q(\bm n)=0$ has no solutions on $\Q^2$ except for $\bm n=\bm0$. Then
\begin{equation*}
\wh{\mT}_{\bm\mu}(\t) = \mT_{\bm\mu}(\t) + \varphi_{\bm{\mu}}^{[\bm{c_1}]}(\t) - \varphi_{\bm\mu}^{[\bm{c_2}]}(\t),
\end{equation*}
where
\begin{equation*}
\varphi_{\bm\mu}^{[\bm{c_0}]}(\t) := \sqrt{\t_2}\sum_{\bm n\in\Z^2+\bm\mu} \a_{t_0} \left(\bm n\sqrt{\t_2}\right) q^{Q(\bm n)}
\end{equation*}
with
\begin{equation*}
\alpha_{t_0}(\bm n) :=
\begin{cases}
\int_{t_0}^{\infty} e^{-\pi B(\bm n, \bm{c(t)})^2}dt \quad & \textnormal{if } B(\bm n,\bm{c_0})B\left(\bm n,\bm{c_0^\perp}\right)>0, \\
-\int_{-\infty}^{t_0} e^{-\pi B(\bm n, \bm{c(t)})^2}dt
\quad & \textnormal{if } B(\bm n,\bm{c_0})B\left(\bm n,\bm{c_0^\perp}\right)<0, \\
0 \quad & \textnormal{otherwise}.
\end{cases}
\end{equation*}
\end{prop}

In summary, we can follow the nomenclature for mock modular forms and refer to $\varphi_{\bm\mu}^{[\bm{c_1}]}$ and $\varphi_{\bm\mu}^{[\bm{c_2}]}$ as {\it shadow contributions} that correct the modular behavior of the mock Maass theta function $\mT_{\bm\mu}$. They do however break the property of being a Laplacian eigenfunction (in a way similar to the breaking of holomorphy property by
completions of mock modular forms).

Now for our purposes, the property of being a Laplacian eigenfunction is crucial for making contact with $q$-series such as \eqref{eq:defn_mock_maass_q_series} that appear in our treatment of the functions $L_1,\dots,L_{12}$. In turn, to prove modular properties for (linear combinations of) $\mt_{\bm\mu}$ as in Section \ref{sec:maass_forms}, we need the corresponding shadow contributions to $\wh{\mT}_{\bm\mu}$ to vanish. One criterion that can be used to prove such a statement is given by Lemma 5.1 of \cite{Zw}. It states that if $\g\in\GL_2(\Z)$ is such that $\g^TA\g=A$, $\det(\g)=+1$, and $\g C_Q = C_Q$, then we have
\begin{equation*}
\varphi_{\g\bm a}^{[\g \bm{c}]}(\t) = \varphi_{\bm a}^{[\bm{c}]}(\t).
\end{equation*}
In our work, we need a modified version of this criterion, which we state and prove next.
\begin{lemma}\label{lemma:sig}
Let  $Q(\bm n)=\frac12\bm n^TA\bm n$ be an even quadratic form of signature $(1,1)$, let $\bm{c_3},\bm{c_4}\in C_Q$, and let $\g\in\GL_2(\Z)$ be such that $\g^TA\g=A$, $\det(\g)=-1$, and $\g\bm{c_3}=\pm\bm{c_4}$. Then 
\begin{equation*}
\varphi_{\g\bm a}^{[\bm{c_4}]}(\t) = -\varphi_{\bm a}^{[\bm{c_3}]}(\t).
\end{equation*}
\end{lemma}

\begin{proof}
We first assume that $\g \bm{c_3} = - \bm{c_4}$ and note that it is not hard to see that 
\begin{equation}\label{eq:gamma}
\g \bm{c(t)}
=
-\bm{c(t_3+t_4-t)}.
\end{equation}
Then we have (by changing variables as $t \mapsto t_3+t_4-t$ and using \eqref{eq:gamma} and that $\g^TA\g=A$)
\begin{equation*}
\int_{t_4}^\infty e^{-\pi B(\g\bm n,\bm{c(t)})^2\t_2} dt 
= \int_{-\infty}^{t_3} e^{-\pi B(\g\bm n,\bm{c(t_3+t_4-t)})^2\t_2} dt 
= \int_{-\infty}^{t_3} e^{-\pi B(\bm n,\bm{c(t)})^2\t_2} dt.
\end{equation*}
Similarly, we have
\begin{equation*}
\int_{-\infty}^{t_4} e^{-\pi B(\g \bm n, \bm{c(t)})^2\t_2}dt 
= 
\int_{t_3}^{\infty} e^{-\pi B(\g \bm n, \bm{c(t_3+t_4-t)})^2\t_2}dt
=
\int_{t_3}^{\infty} e^{-\pi B(\bm n, \bm{c(t)})^2\t_2}dt.
\end{equation*}
Using these identities with the fact that
\begin{equation*}
B(\g\bm n,\bm{c_4}) B\left(\g\bm n,\bm{c_4^\perp}\right) = B(\g\bm n,-\g\bm{c_3}) B\left(\g\bm n,\g\bm{c_3^\perp}\right) = -B(\bm n,\bm{c_3}) B\left(\bm n,\bm{c_3^\perp}\right),
\end{equation*}
we immediately see that $\a_{t_4}(\g\bm n\sqrt{\t_2})=-\a_{t_3}(\bm n\sqrt{\t_2})$ and hence that $\varphi_{\g\bm a}^{[\bm{c_4}]}(\t)=-\varphi_{\bm a}^{[\bm{c_3}]}(\t)$.

For the case $\g\bm{c_3}=\bm{c_4}$, we can write $\g = (-\g) (-I_2)$, where $-\g$ satisfies the hypotheses of the lemma with $(-\g)\bm{c_3}=-\bm{c_4}$. Then the first part of the lemma, together with the trivial identity $\varphi_{-\bm{a}}^{[\bm{c_3}]}(\tau) = \varphi_{\bm{a}}^{[\bm{c_3}]}(\tau)$ (see Lemma 5.1 of \cite{Zw}) implies the result.
\end{proof}

\section{The functions $L_1$\textendash$L_4$}\label{sec:L1_L4_analysis}

In this section, we begin our investigation of the twelve functions introduced by Lovejoy and Osburn \cite{LoOs}. We start with the functions $L_1,\dots,L_4$ that count ideals in the ring of integers of $\Q(\sqrt{2})$.

\subsection{Rewriting as theta functions}\label{sec:L1L4_rearrangement}
We start with the following lemma, which rewrites $L_1,\dots,L_4$ in terms of theta functions as in equation \eqref{eq:indefinite_false_theta}.

\begin{lem}\label{lem:l1l4}
We have
\begin{align*}
q^{-\frac{17}{32}}L_1(q) &= \frac12\left(\sum_{\bm n\in\Z^2+\left(\frac{1}{16}\ \frac38\right)^T} + \sum_{\bm n\in\Z^2+\left(\frac{7}{16}\ \frac18\right)^T}\right) (1+\sgn(n_1+n_2)\sgn(n_1-n_2)) q^{8n_1^2-4n_2^2},\\
q^{\frac{7}{32}}L_2(q) &= \frac12\left(\sum_{\bm n\in\Z^2+\left(\frac{3}{16}\ \frac18\right)^T} + \sum_{\bm n\in\Z^2+\left(\frac{5}{16}\ \frac38\right)^T}\right) (1+\sgn(n_1+n_2)\sgn(n_1-n_2)) q^{8n_1^2-4n_2^2},\\
q^{-\frac{33}{32}}L_3(q) &= \frac12\left(\sum_{\bm n\in\Z^2+\left(\frac{1}{16}\ \frac18\right)^T} + \sum_{\bm n\in\Z^2+\left(\frac{7}{16}\ \frac38\right)^T}\right) (1+\sgn(n_1+n_2)\sgn(n_1-n_2)) q^{8n_1^2-4n_2^2},\\
q^{-\frac{9}{32}}L_4(q) &= \frac12\left(\sum_{\bm n\in\Z^2+\left(\frac{3}{16}\ \frac38\right)^T} + \sum_{\bm n\in\Z^2+\left(\frac{5}{16}\ \frac18\right)^T}\right) (1+\sgn(n_1+n_2)\sgn(n_1-n_2)) q^{8n_1^2-4n_2^2}.
\end{align*}
\end{lem}

\begin{proof} We start with the following identity given in the proof of Theorem 1.1 of \cite{Lo}
\begin{multline*}
q^{-17}L_1\left(q^{32}\right) = \sum_{\substack{n\ge1\\-n\le j\le n-1}} \left(q^{(16n-1)^2-2(8j+3)^2}+q^{(16n+1)^2-2(8j+3)^2}\right)\\
+ \sum_{\substack{n\ge0\\-n\le j\le n}} \left(q^{(16n+7)^2-2(8j+1)^2}+q^{(16n+9)^2-2(8j+1)^2}\right).
\end{multline*}
First note that the conditions $n\ge1$ and $n\ge0$ can be dropped with the understanding that we are summing over all pairs $(n,j)\in\Z^2$ satisfying the conditions $-n\le j\le n-1$ and $-n\le j\le n$, respectively. We then change $n\mapsto-n$ in the first sum and $n\mapsto-n-1$ in the forth sum to obtain the claim for $L_1 (q)$.

As the proof of the remaining identities are similar, we just state the related identities from the proof of Theorem 1.1 of \cite{Lo}:
\begin{multline*}
q^7L_2\left(q^{32}\right) = \sum_{\substack{n\ge0\\-n\le j \le n}} \left(q^{(16n+3)^2-2(8j+1)^2}+q^{(16n+13)^2-2(8j+1)^2}\right)\\
+ \sum_{\substack{n\ge0\\-n-1\le j\le n}} \left(q^{(16n+21)^2-2(8j+3)^2}+q^{(16n+11)^2-2(8j+3)^2}\right),
\end{multline*}
\begin{multline*}
q^{-33}L_3\left(q^{32}\right) = \sum_{\substack{n\ge1\\-n\le j\le n-1}} \left(q^{(16n-1)^2-2(8j+1)^2}+q^{(16n+1)^2-2(8j+1)^2}\right)\\
+ \sum_{\substack{n\ge0\\-n\le j\le n}} \left(q^{(16n+7)^2-2(8j+3)^2}+q^{(16n+9)^2-2(8j+3)^2}\right),
\end{multline*}
\begin{multline*}
q^{-9}L_4\left(q^{32}\right) = \sum_{\substack{n\ge0\\-n\le j\le n-1}} q^{(16n+3)^2-2(8j+3)^2} + \sum_{\substack{n\ge0\\-n-1\le j\le n}} q^{(16n+13)^2-2(8j+3)^2}\\
+ \sum_{\substack{n\ge-1\\-n-1\le j\le n+1}} q^{(16n+21)^2-2(8j+1)^2} + \sum_{\substack{n\ge0\\-n\le j\le n}} q^{(16n+11)^2-2(8j+1)^2}.\qedhere
\end{multline*}
\end{proof}

To express these functions more compactly, we define the quadratic form
\begin{equation}\label{eq:quadratic_form_example1}
Q(\bm n):=8n_1^2-4n_2^2
\end{equation}
and the vectors 
\begin{equation}\label{eq:cj_vector_example1}
\bm{c_1}:=\frac{1}{2\sqrt{2}}(-1\ \ 2)^T
\andd
\bm{c_2}:=\frac{1}{2\sqrt{2}}(1\ \ 2)^T
\end{equation}
in $C_Q$ (as defined in Remark \ref{rem:P_choice}).
For these choices, the theta function in equation \eqref{eq:defn_mock_maass_q_series} is
\begin{align*}\label{eq:vartheta_example1}
\mt_{\bm\mu}(\t) 
&= \frac12\sum_{\bm n\in\Z^2+\bm\mu} (1+\sgn(n_1+n_2)\sgn(n_1-n_2)) q^{8n_1^2-4n_2^2}
\quad
\mbox{for } \bm{\mu} \not\in \Z^2.
\end{align*}
We can now express the theta functions found in Lemma \ref{lem:l1l4} in terms of these $\mt_{\bm{\mu}}$'s.
\begin{lem}\label{lem:l1l4rew}
We have
\begin{equation*}
q^{-\frac{17}{32}}L_1(q) = \mf_3 (\t),  \quad
q^\frac{7}{32}L_2(q) = \mf_1 (\t), \quad
q^{-\frac{33}{32}}L_3(q) = \mf_0 (\t), \quad
q^{-\frac{9}{32}} L_4 (q) = \mf_2 (\t),
\end{equation*}
where
\begin{equation*}
\mf_j (\t) := \mt_{\bm{\mu_j}}(\t) + \mt_{\bm{\mu_j}+\bm\l}(\t) 
\quad \mbox{with }
\bm{\mu_j}:=\lp \frac{2j+1}{16}\ \ \frac18 \rp^T \mbox{ and } 
\bm\l:=\lp \frac12\ \ \frac12 \rp^T.
\end{equation*}
\end{lem}

\begin{proof}
The claim follows from the relations
\begin{equation*}
\mt_{\bm\mu} = \mt_{-\bm\mu} = \mt_{(\mu_1\ -\mu_2)^T}. \qedhere
\end{equation*}
\end{proof}

\subsection{The corresponding Maass forms}

Next we follow \eqref{eq:defn_mock_maass_theta_fnc} and define the mock Maass theta functions $\mT_{\bm\mu}$ for the quadratic form $Q$ in \eqref{eq:quadratic_form_example1} and vectors $\bm{c_1}$, $\bm{c_2}$ in \eqref{eq:cj_vector_example1}.
Noting that the vectors in $C_Q^\perp$ corresponding to $\bm{c_1}$ and $\bm{c_2}$ are $\bm{c_1^\perp}=\frac12(1\ -1)^T$ and $\bm{c_2^\perp}=\frac12(1\ \ 1)^T$, we find (for $\bm\mu\not\in\Z^2$)
\begin{align*}
\mT_{\bm\mu}(\t) &= \frac{\sqrt{\t_2}}{2}\sum_{\bm n\in\Z^2+\bm\mu} (1+\sgn(n_1+n_2)\sgn(n_1-n_2)) K_0\left(2\pi\left(8n_1^2-4n_2^2\right)\t_2\right) e^{2\pi i\left(8n_1^2-4n_2^2\right)\t_1}\\
&\hspace{.5cm}+ \frac{\sqrt{\t_2}}{2}\sum_{\bm n\in\Z^2+\bm\mu} (1-\sgn(2n_1+n_2)\sgn(2n_1-n_2)) K_0\left(-2\pi\left(8n_1^2-4n_2^2\right)\t_2\right) e^{2\pi i\left(8n_1^2-4n_2^2\right)\t_1} .
\end{align*}
We also have their modular completions given by equation \eqref{eq:mock_maass_theta_fnc_completion} and Proposition \ref{prop:quadzero}:
\begin{equation*}
\wh{\mT}_{\bm\mu} = \mT_{\bm\mu}+ \varphi_{\bm\mu}^{[\bm{c_1}]}  - \varphi_{\bm\mu}^{[\bm{c_2}]}.
\end{equation*}
In the following it is also useful to note the trivial identities 
\begin{equation}\label{eq:F_L1_4_identifications}
\mT_{\bm\mu} = \mT_{-\bm\mu} = \mT_{{(-\mu_1\ \mu_2)^T}} = \mT_{{(\mu_1\ -\mu_2)^T}},\qquad \wh{\mT}_{\bm\mu} = \wh{\mT}_{-\bm\mu} = \wh{\mT}_{{(-\mu_1\ \mu_2)^T}} = \wh{\mT}_{{(\mu_1\ -\mu_2)^T}}.
\end{equation}

Our main interest is of course on the linear combinations of $\mt_{\bm\mu}$'s given in Lemma \ref{lem:l1l4rew}. So for $j\in\{0,1,2,3\}$ we form the corresponding linear combinations of the $\mT_{\bm\mu}$'s and $\wh{\mT}_{\bm\mu}$'s to define
\begin{equation*}
\mF_j := \mT_{\bm{\mu_j}} + \mT_{\bm{\mu_j}+\bm\l}
\andd
\wh{\mF}_j := \wh{\mT}_{\bm{\mu_j}} + \wh{\mT}_{\bm{\mu_j}+\bm\l}.
\end{equation*}
With the next lemma, we show that the shadow contributions to $\wh{\mF}_j$ do in fact vanish and hence the $\mF_j$'s constitute a Maass form.\footnote{Under any element of $\SL_2(\Z)$, the functions $\wh F_j$ transform into a linear combination of completed mock Maass theta functions $\wh\mT_{\bm\mu}$ according to Theorem \ref{thm:That_general_modular_transformations}. According to the Section 3 of \cite{Zw}, we have 
\[
	\left|q^{Q(\bm n)}\int_{t_1}^{t_2} e^{-\pi B(\bm n,\bm{c(t)})^2\t_2} dt \right| \le |t_2-t_1|e^{-C\t_2\left(n_1^2+n_2^2\right)}
\]
for any $\bm n$, with a constant $C$ depending on $\bm{c_1}$ and $\bm{c_2}$. So the infinite series in equation \eqref{eq:mock_maass_theta_fnc_completion} can be bounded from above with an ordinary theta function and this ensures that each $\wh F_j$ has at most polynomial growth at the cusps. This allows $F_j=\wh F_j$ to satisfy the growth condition required from Maass forms in Definition \ref{defn:Maass_form}. The same argument also applies to the relevant functions for $L_5,\dots,L_8$ and $L_9,\dots,L_{12}$ cases below.}

\begin{lem}\label{lem:modular_obstruction_cancellation}
For $j\in\{0,1,2,3\}$ we have
\begin{equation*}
\mF_j  = \wh{\mF}_j  .
\end{equation*}
\end{lem}

\begin{proof}
Suppose that $\g\in\GL_2(\Z)$ is an automorphism such that $\g^TA\g=A$, $\det(\g)=-1$, and $\g\bm{c}= \pm \bm{c}$ for some $\bm{c} \in C_Q$. Suppose moreover that $\g\bm{\l}\equiv\bm{\l}\pmod{\Z^2}$ and that $\g$ transforms $\bm{\mu}$ to $\pm \bm{\mu}$ or $\pm (\bm{\mu} + \bm{\l})$ modulo $\Z^2$. Then the identity $\varphi_{-\bm a}^{[\bm c]}=\varphi_{\bm a}^{[\bm c]}$ and Lemma \ref{lemma:sig} yield
\begin{equation*}
\varphi_{\bm{\mu}}^{[\bm{c}]} + \varphi_{\bm{\mu}+\bm\l}^{[\bm{c}]} = 0 .
\end{equation*}

We use this fact with the automorphisms $\g_1:=\pmat{-3&-2\\4&3}$ and $\g_2:=\pmat{3&-2\\4&-3}$, which both satisfy $\det(\g_j)=-1$ and $\g^TA\g=A$, as well as $\g_1\bm{c_1}=\bm{c_1}$ and $\g_2\bm{c_2}=-\bm{c_2}$. Since
\begin{equation*}
\g_1\bm{\mu_j} = \mat{\frac{-6j-7}{16}&\frac{4j+5}{8}}^T,\qquad \g_2\bm{\mu_j} = \mat{\frac{6j-1}{16}&\frac{4j-1}{8}}^T,\qquad \g_1\bm\l \equiv \g_2\bm\l \equiv \bm\l \pmod{\Z^2},
\end{equation*}
we also find that $\g_1$ and $\g_2$ transform $\bm{\mu_j}$ to $\pm \bm{\mu_j}$ or $\pm (\bm{\mu_j}+\bm\l)$ modulo $\Z^2$ to conclude that
\begin{equation*}
\wh{\mF}_j - \mF_j = \lp \varphi_{\bm{\mu_j}}^{[\bm{c_1}]} + \varphi_{\bm{\mu_j}+\bm\l}^{[\bm{c_1}]} \rp
- \lp \varphi_{\bm{\mu_j}}^{[\bm{c_2}]} + \varphi_{\bm{\mu_j}+\bm\l}^{[\bm{c_2}]} \rp = 0
\quad \mbox{for all } j. \qedhere
\end{equation*}
\end{proof}

We next state the details of modular transformations for the $\mF_j$'s under the modular group $\GG_0(2)$, which is generated by $T:=\pmat{1&1\\0&1}$, $R:=\pmat{1&0\\2&1}$, and $-I$. Here and throughout we use $\z_n:=e^\frac{2\pi i}{n}$.

\begin{prop}\label{prop:L_1_4_modular_transformation}
The functions $\mF_j$ for $j\in\{0,1,2,3\}$ transform like a vector-valued modular function under $\GG_0(2)$:
\begin{equation*}
\mF_j\left(\frac{a\t+b}{c\t+d}\right) = \sum_{k=0}^3 \LL_M(j,k)\mF_k(\t) \qquad \mbox{for all } M \in \GG_0(2),
\end{equation*}
where the multiplier system $\LL$ is as follows for $T=\pmat{1&1\\0&1}$ and $R=\pmat{1&0\\2&1}$:
\begin{align*}
\LL_T = \mat{\z_{32}^{31}&0&0&0\\0&\z_{32}^7&0&0\\0&0&\z_{32}^{23}&0\\0&0&0&\z_{32}^{15}},\ \LL_R = \frac{1}{\sqrt{2}} \mat{\z^{31}_{32}\cos\left(\frac{3\pi}{16}\right) & \z^3_{32}\cos\left(\frac{\pi}{16}\right) & \z^{11}_{32}\sin\left(\frac{\pi}{16}\right) & \z^{23}_{32}\sin\left(\frac{3\pi}{16}\right) \\\vspace{-.3cm}\\ \z^3_{32}\cos\left(\frac{\pi}{16}\right) & \z^{23}_{32}\sin\left(\frac{3\pi}{16}\right) & \z^{31}_{32}\cos\left(\frac{3\pi}{16}\right) & \z^{11}_{32}\sin\left(\frac{\pi}{16}\right) \\\vspace{-.3cm}\\ \z^{11}_{32}\sin\left(\frac{\pi}{16}\right) & \z^{31}_{32}\cos\left(\frac{3\pi}{16}\right) & \z^{23}_{32}\sin\left(\frac{3\pi}{16}\right) & \z^3_{32}\cos\left(\frac{\pi}{16}\right) \\\vspace{-.3cm}\\ \z^{23}_{32}\sin\left(\frac{3\pi}{16}\right) & \z^{11}_{32}\sin\left(\frac{\pi}{16}\right) & \z^3_{32}\cos\left(\frac{\pi}{16}\right) & \z^{31}_{32}\cos\left(\frac{3\pi}{16}\right)}.
\end{align*}
\end{prop}

\begin{proof}
The behavior under translation follows from Theorem \ref{thm:That_general_modular_transformations} and the fact that $Q(\bm{\mu_j}+\bm\l) = Q(\bm{\mu_j})+j+1$.
We next consider the transformation $\t\mapsto\frac{\t}{2\t+1}$. By Theorem \ref{thm:That_general_modular_transformations}, we have
\begin{equation*}
\wh{\mT}_{\bm\mu}\left(\frac{\t}{2\t+1}\right) = \sum_{\bm\nu\in A^{-1}\Z^2/\Z^2} \psi(\bm\mu,\bm\nu)\wh{\mT}_{\bm\nu}(\t),
\end{equation*}
where
\begin{equation*}
\psi(\bm\mu,\bm\nu) = \frac{1}{2\sqrt{|\det(A)|}} \sum_{\bm m\in\Z^2/2\Z^2} e^{\pi i(Q(\bm m+\bm\mu)-B(\bm m+\bm\mu,\bm\nu)+Q(\bm\nu))}.
\end{equation*}
Since $Q(\bm m+\bm\mu)-B(\bm m+\bm\mu,\bm\nu)+Q(\bm\nu)=Q(\bm m)+B(\bm m,\bm\mu-\bm\nu)+Q(\bm\mu-\bm\nu)$ and $Q(\bm m)=8m_1^2-4m_2^2$ is even we get
\begin{equation*}
\psi(\bm\mu,\bm\nu) =  \frac{e^{\pi iQ(\bm\mu-\bm\nu)}}{16\sqrt{2}} \sum_{\bm m\in\Z^2/2\Z^2} e^{\pi iB(\bm m,\bm\mu-\bm\nu)}.
\end{equation*}
The elements of $A^{-1}\Z^2/\Z^2$ are of the form $(\frac{r_1}{16}\ \ \frac{r_2}{8})^T$ with $r_1\in\{0,1,2\dots,15\}$ and $r_2\in\{0,1,2\dots,7\}$. Letting $\bm\mu=(\frac{2r_1+1}{16}\ \ \frac{2r_2 +1}{8})^T$ and $\bm\nu=(\frac{2s_1+\ell_1}{16}\ \ \frac{2s_2+\ell_2}{8})^T$, where $r_1,s_1\in\{0,1,\dots,7\}$, $r_2,s_2\in\{0,1,2,3\}$, and $\ell_1,\ell_2\in\{0,1\}$, we find
\begin{equation*}
\sum_{\bm m\in\Z^2/2\Z^2} e^{\pi iB(\bm m,\bm\mu-\bm\nu)} = 4\d_{\ell_1,1}\d_{\ell_2,1}.
\end{equation*}
In particular, this shows that the $\wh{\mT}_{\bm\mu}$'s with $\bm\mu$ of the form $(\frac{2r_1+1}{16}\ \ \frac{2r_2+1}{8})^T$ transform among each other. Moreover, because of equation \eqref{eq:F_L1_4_identifications}, changing $\mu_1$ from $\frac{1}{16}, \frac{3}{16}, \frac{5}{16}, \frac{7}{16}$ to $\frac{15}{16}, \frac{13}{16}, \frac{11}{16}, \frac{9}{16}$, respectively, (or vice versa), or changing $\mu_2$ from $\frac{1}{8}, \frac{5}{8}$ to $\frac{7}{8},\frac{3}{8}$, respectively, does not change $\wh{\mT}_{\bm{\mu}}$. So we can combine the contributions of cosets corresponding to $\bm{\mu}, -\bm{\mu}, (-\mu_1\ \ \mu_2)^T, (\mu_1\ -\mu_2)^T$ by restricting $\bm{\mu}$ to
\begin{equation*}
\calS := \{\bm{\mu_0},\bm{\mu_1},\bm{\mu_2},\bm{\mu_3},\bm{\mu_0}+\bm\l,\bm{\mu_1}+\bm\l,\bm{\mu_2}+\bm\l,\bm{\mu_3}+\bm\l\},
\end{equation*}
and by combining the corresponding factors in the transformation as
\begin{equation*}
\rho (\bm{\mu}, \bm{\nu}) := 
\psi (\bm{\mu}, \bm{\nu}) + \psi (\bm{\mu}, -\bm{\nu})
+ \psi (\bm{\mu}, (\nu_1, -\nu_2)) + \psi (\bm{\mu}, (-\nu_1, \nu_2)) 
\andd
\psi (\bm{\mu}, \bm{\nu}) =  \frac{e^{\pi i Q(\bm{\mu}-\bm{\nu}) }}{4 \sqrt{2}}.
\end{equation*}
This yields
\begin{equation*}
\wh{\mT}_{\bm{\mu}} \lp \frac{\t}{2 \t+1} \rp
=
\sum_{\bm{\nu} \in \mathcal{S}} 
\rho (\bm{\mu}, \bm{\nu}) \wh{\mT}_{\bm{\nu}} (\tau) 
\quad \mbox{for } \bm{\mu} \in  \mathcal{S},
\end{equation*}
where we simplify $\rho$ as
\begin{equation*}
\rho(\bm\mu,\bm\nu) = \frac{e^{\pi i(Q(\bm\mu)+Q(\bm\nu))}}{\sqrt{2}} \cos(16\pi\mu_1\nu_1) \cos(8\pi\mu_2\nu_2).
\end{equation*}

Since our ultimate goal is to study the transformation of $\wh{\mF}_j$, we write
\begin{equation*}
\wh{\mT}_{\bm{\mu}} \lp \frac{\t}{2 \t+1} \rp + \wh{\mT}_{\bm{\mu} + \bm{\l}} \lp \frac{\t}{2 \t+1} \rp
=
\sum_{\bm{\nu} \in \mathcal{S}} 
\l (\bm{\mu}, \bm{\nu}) \wh{\mT}_{\bm{\nu}} (\tau),
\end{equation*}
where
\begin{equation*}
\l (\bm{\mu}, \bm{\nu})  := \rho (\bm{\mu}, \bm{\nu})  + \rho (\bm{\mu}+\bm{\l}, \bm{\nu}) .
\end{equation*}
Since $Q(\bm\l)=1$, $B(\bm\mu,\bm\l)=8\mu_1-4\mu_2$, $\mu_2\in\{\frac18,\frac58\}$ for $\bm{\mu} \in \mathcal{S}$, we find the identity 
\begin{equation*}
e^{\pi iQ(\bm\mu+\bm\l)} = ie^{\pi iQ(\bm\mu)}e^{8\pi i\mu_1}
\quad \mbox{for } \bm{\mu} \in \mathcal{S},
\end{equation*}
which we can use to show
\begin{equation*}
\l(\bm\mu,\bm\nu+\bm\l) = \l(\bm\mu,\bm\nu).
\end{equation*}
So we finally get
\begin{equation*}
\wh{\mT}_{\bm{\mu_j}}\left(\frac{\t}{2\t+1}\right) + \wh{\mT}_{\bm{\mu_j}+\bm\l}\left(\frac{\t}{2\t+1}\right) = \sum_{k=0}^3 \l(\bm{\mu_j},\bm{\mu_k}) \left(\wh{\mT}_{\bm{\mu_k}}(\t)+\wh{\mT}_{\bm{\mu_k}+\bm\l}(\t)\right)
\end{equation*}
and hence
\begin{equation*}
\mF_j\left(\frac{\t}{2\t+1}\right) = \sum_{k=0}^3 \l(\bm{\mu_j},\bm{\mu_k}) \mF_k(\t).
\end{equation*}
Computing the explicit values of $\l(\bm{\mu_j},\bm{\mu_k})$ gives the stated transformation.
\end{proof}

\subsection{Quantum modularity}

To summarize our findings so far, we know from Proposition \ref{prop:L_1_4_modular_transformation} that the functions $\mF_j$ for $j\in\{0,1,2,3\}$ form a vector-valued Maass form for $\GG_0(2)$. Our next goal is to apply the results of Section \ref{sec:maass_forms} on these functions. Firstly, we note that thanks to equation \eqref{eq:defn_u_and_error} and Proposition \ref{prop:uj_fourier_expansion}, the functions $\mf_j$ (which are basically equal to the functions $L_1,\dots,L_4$ thanks to Lemma \ref{lem:l1l4rew}) are related to the $\mF_j$'s as
\begin{equation*}
\mf_j(\t) = -\frac2\pi\int_\t^{i\infty} [\mF_j(z),R_\t(z)] .
\end{equation*}
Then thanks to  Proposition \ref{prop:uj_modularity}, they satisfy the following modular transformation for any $M:=\pmat{a&b\\c&d}\in\GG_0(2)$ and $\t\in\H$ with $\t_1\ne-\frac dc$:
\begin{equation*}
\mf_j\left(\frac{a\t+b}{c\t+d}\right) = \sgn(c\t_1+d)(c\t+d) \sum_{k=0}^3 \LL_M(j,k)\left(\mf_k(\t)+\mcalF_{k,-\frac dc}(\t)\right).
\end{equation*}
Here the obstruction to modularity $\mcalF_{k,-\frac dc}$ is a holomorphic function on $\C\setminus(-\frac dc+i\R)$ (thanks to Proposition \ref{prop:modularity_error_holomorphy}) defined as in equation \eqref{eq:defn_u_and_error}:
\begin{equation*}
\mcalF_{j,\varrho}(\t) := \frac2\pi\int_\varrho^{i\infty} [\mF_j(z),R_\t(z)] .
\end{equation*}

Next, following Proposition \ref{prop:uj_vertical_limit} we define the functions $\mfrakf_j : \calQ_{\GG_0 (2)} \to \C$ with $j\in\{0,1,2,3\}$ by 
\begin{equation}\label{eq:L1_4_quantum_limit_defn}
\mfrakf_j(x) := \lim_{t\to0^+} \mf_j(x+it) .
\end{equation}
Then finally by Theorem \ref{thm:frak_uj_quantum_modular_transformations} we find the quantum modular properties for the rational limits of $L_1, \ldots, L_4$.
\begin{prop}\label{prop:L1_4_quantum_modular}
The functions $\mfrakf_j : \calQ_{\GG_0 (2)} \to \C$ with $j\in\{0,1,2,3\}$ form a vector-valued quantum modular form transforming as follows for any $M := \pmat{a&b\\c&d}\in \GG_0 (2)$ and $x \in \mathcal{Q}_{\GG_0 (2)} \setminus \{-\frac{d}{c} \}$:
\begin{equation*}
\mfrakf_j\lp \frac{ax+b}{cx+d}\rp 
= |cx+d| \sum_{k=0}^3 \LL_M(j,k) 
\lp \mfrakf_k(x)+\mcalF_{k,-\frac{d}{c}}(x)\rp .
\end{equation*}
Here the multiplier system $\LL$ is as given in Proposition \ref{prop:L_1_4_modular_transformation}.
\end{prop}

\section{The functions $L_5$\textendash$L_8$}\label{sec:L5_L8_analysis}
We continue in this section with the functions $L_5,\dots,L_8$ that count ideals in the ring of integers of $\Q(\sqrt{3})$. This is possibly the most interesting family of functions among our examples since its discussion requires the novel features studied in Section \ref{sec:maass_forms} due to the presence of constant terms.

\subsection{Rewriting as theta functions}

We begin as in Section \ref{sec:L1L4_rearrangement} by rewriting our functions in terms of theta functions.

\begin{lem}\label{lem:l5l8}
We have
\begin{align*}
q^{-1}L_5(q) &= \frac12\left(\sum_{\bm n\in\Z^2\setminus\{\bm 0\}} + \sum_{\bm n\in\Z^2+\left(\frac12\ \frac12\right)^T}\right) (1+\sgn(n_1+n_2)\sgn(n_1-n_2)) q^{6n_1^2-2n_2^2},\\
q^{-\frac12}L_6(q) &= \frac12\left(\sum_{\bm n\in\Z^2+\left(0\ \frac12\right)^T} + \sum_{\bm n\in\Z^2+\left(\frac12\ 0\right)^T}\right) (1+\sgn(n_1+n_2)\sgn(n_1-n_2)) q^{6n_1^2-2n_2^2},\\
q^\frac16L_7(q) &= \frac12\left(\sum_{\bm n\in\Z^2+\left(\frac13\ \frac12\right)^T} + \sum_{\bm n\in\Z^2+\left(\frac16\ 0\right)^T}\right) (1+\sgn(n_1+n_2)\sgn(n_1-n_2)) q^{6n_1^2-2n_2^2},\\
q^{-\frac13}L_8(q) &= \frac12\left(\sum_{\bm n\in\Z^2+\left(\frac16\ \frac12\right)^T} + \sum_{\bm n\in\Z^2+\left(\frac13\ 0\right)^T}\right) (1+\sgn(n_1+n_2)\sgn(n_1-n_2)) q^{6n_1^2-2n_2^2}.
\end{align*}
\end{lem}

\begin{proof}
As the proof is analogous to that of Lemma \ref{lem:l1l4} we only state the identities used from Theorem 1.2 of \cite{Lo}. These are
\begin{align*}
q^{-2}L_5\left(q^2\right) &= 2\sum_{\substack{n\ge1\\-n\le j\le n-1}} q^{3(2n)^2-(2j)^2} + 2\sum_{\substack{n\ge0\\-n\le j\le n}} q^{3(2n+1)^2-(2j+1)^2},\\
q^{-1}L_6\left(q^2\right) &= 2\sum_{\substack{n\ge1\\-n\le j\le n-1}} q^{3(2n)^2-(2j+1)^2} + 2\sum_{\substack{n\ge0\\-n\le j\le n}} q^{3(2n+1)^2-(2j)^2},\\
qL_7\left(q^6\right) &= \sum_{\substack{n\ge0\\-n-1\le j\le n}} \left(q^{(6n+8)^2-3(2j+1)^2}+q^{(6n+4)^2-3(2j+1)^2}\right)\\
&\hspace{4.5cm}+ \sum_{\substack{n\ge0\\-n\le j\le n}} \left(q^{(6n+1)^2-3(2j)^2}+q^{(6n+5)^2-3(2j)^2}\right),\\
q^{-2}L_8\left(q^6\right) &=\sum_{\substack{n\ge0\\-n\le j\le n-1}} q^{(6n+1)^2-3(2j+1)^2} + \sum_{\substack{n\ge0\\-n-1\le j\le n}} q^{(6n+5)^2-3(2j+1)^2}\\
&\hspace{4.5cm}+ \sum_{\substack{n\ge-1\\-n-1\le j\le n+1}} q^{(6n+8)^2-3(2j)^2} + \sum_{\substack{n\ge0\\-n\le j\le n}} q^{(6n+4)^2-3(2j)^2}. \qedhere
\end{align*}
\end{proof}

To rewrite these expressions more compactly, we define the quadratic form
\begin{equation}\label{eq:quadratic_form_example2}
Q(\bm n):=6n_1^2-2n_2^2
\end{equation}
and the vectors 
\begin{equation}\label{eq:cj_vector_example2}
\bm{c_1}:=\frac{1}{2\sqrt{3}}(-1\ \ 3)^T
\andd
\bm{c_2}:=\frac{1}{2\sqrt{3}}(1\ \ 3)^T
\end{equation}
in $C_Q$ (as defined in Remark \ref{rem:P_choice}). The parameters $t_1$ and $t_2$ that describe these vectors according to \eqref{eq:ct_definition} are 
\begin{equation*}
t_1 = -\arctanh\left(\frac{1}{\sqrt{3}}\right) \andd t_2 = \arctanh\left(\frac{1}{\sqrt{3}}\right).
\end{equation*}
According to these values, the theta function given in equation \eqref{eq:defn_mock_maass_q_series} becomes
\begin{align*}\label{eq:vartheta_example2}
\mt_{\bm\mu}(\t) 
&= - \frac{2}{\pi} \arctanh \lp \frac{1}{\sqrt{3}} \rp  \d_{\bm\mu\in\Z^2}  +
\frac12\sum_{\substack{\bm n\in \Z^2+\bm\mu \\ \bm{n} \neq \bm{0} }}    (1+\sgn(n_1+n_2)\sgn(n_1-n_2)) q^{6n_1^2-2n_2^2}.
\end{align*}
We can now rewrite the expressions in Lemma \ref{lem:l5l8} in terms of these $\mt_{\bm{\mu}}$'s.
\begin{lem}\label{lem:l5l8rew}
We have
\begin{equation*}
q^{-1}L_5(q) - \frac{2}{\pi} \arctanh \lp \frac{1}{\sqrt{3}} \rp = \mg_0 (\t),  \ \ 
q^{-\frac12}L_6(q)  = \mg_3 (\t), \ \ 
q^\frac16L_7(q)  = \mg_1 (\t),\ \ 
q^{-\frac13}L_8(q)  = \mg_2 (\t),
\end{equation*}
where
\begin{equation*}
\mg_j (\t) := \mt_{\bm{\mu_j}}(\t) + \mt_{\bm{\mu_j}+\bm\l}(\t) 
\quad \mbox{with }
\bm{\mu_j}:=\lp \frac j6\ \ 0 \rp^T \mbox{ and } 
\bm\l:=\lp \frac12\ \ \frac12 \rp^T.
\end{equation*}
\end{lem}

\subsection{The corresponding Maass forms}

We next follow equation \eqref{eq:defn_mock_maass_theta_fnc} and define the mock Maass theta functions $\mT_{\bm\mu}$ given the quadratic form $Q$ in \eqref{eq:quadratic_form_example2} and vectors $\bm{c_1}$, $\bm{c_2}$ in \eqref{eq:cj_vector_example2}. For this purpose, we first find the vectors in $C_Q^\perp$ that correspond to $\bm{c_1}$ and $\bm{c_2}$ as $\bm{c_1^\perp}=\frac12(1\ -1)^T$ and $\bm{c_2^\perp}=\frac12(1\ \ 1)^T$. Then we get
\begin{align*}
\mT_{\bm\mu}(\t) &= \frac{\sqrt{\t_2}}{2} 
\mspace{-8mu}
\sum_{\bm n\in(\Z^2+\bm\mu)\setminus\{\bm0\}}
\mspace{-12mu}
 (1+\sgn(n_1+n_2)\sgn(n_1-n_2)) K_0\left(2\pi\left(6n_1^2-2n_2^2\right)\t_2\right) e^{2\pi i\left(6n_1^2-2n_2^2\right)\t_1}\\
& \ \ + \frac{\sqrt{\t_2}}{2} 
\mspace{-8mu}
\sum_{\bm n\in(\Z^2+\bm\mu)\setminus\{\bm0\}}
\mspace{-12mu}	
 (1-\sgn(3n_1+n_2)\sgn(3n_1-n_2)) K_0\left(-2\pi\left(6n_1^2-2n_2^2\right)\t_2\right) e^{2\pi i\left(6n_1^2-2n_2^2\right)\t_1}\\
& \qquad + 2\arctanh\left(\frac{1}{\sqrt{3}}\right) \sqrt{\t_2}  \d_{\bm\mu\in\Z^2} . 
\end{align*}
We also have the corresponding modular completion given by equation \eqref{eq:mock_maass_theta_fnc_completion} and Proposition \ref{prop:quadzero}:
\begin{equation*}
\wh{\mT}_{\bm\mu} = \mT_{\bm\mu}+ \varphi_{\bm\mu}^{[\bm{c_1}]}  - \varphi_{\bm\mu}^{[\bm{c_2}]}.
\end{equation*}
As before these functions satisfy the following elementary properties:
\begin{equation*}\label{eq:F_L5_8_identifications}
\mT_{\bm\mu} = \mT_{-\bm\mu} = \mT_{{(-\mu_1\ \mu_2)^T}} = \mT_{{(\mu_1\ -\mu_2)^T}},\qquad \wh{\mT}_{\bm\mu} = \wh{\mT}_{-\bm\mu} = \wh{\mT}_{{(-\mu_1\ \mu_2)^T}} = \wh{\mT}_{{(\mu_1\ -\mu_2)^T}}.
\end{equation*}

Since we are interested in the linear combinations of $\mt_{\bm\mu}$'s given in Lemma \ref{lem:l5l8rew}, we make the following  definitions for $j\in\{0,1,2,3\}$:
\begin{equation*}
\mG_j := \mT_{\bm{\mu_j}} + \mT_{\bm{\mu_j}+\bm\l}
\andd
\wh{\mG}_j := \wh{\mT}_{\bm{\mu_j}} + \wh{\mT}_{\bm{\mu_j}+\bm\l}.
\end{equation*}
We now proceed as in Lemma \ref{lem:modular_obstruction_cancellation} to show that the shadow contributions to $\wh{\mG}_j$ vanish. 

\begin{lem}\label{lem:whP}
For $j\in\{0,1,2,3\}$ we have
\begin{equation*}
\mG_j = \wh\mG_j.
\end{equation*}
\end{lem}

Then following the arguments of Proposition \ref{prop:L_1_4_modular_transformation}, we find the modular properties of the Maass forms $\mG_j$.

\begin{prop}\label{prop:L_5_8_modular_transformation}
For $j\in\{0,1,2,3\}$ the functions $\mG_j$ transform like vector-valued modular function under $\GG_0(2)$:
\begin{equation*}
\mG_j\left(\frac{a\t+b}{c\t+d}\right) = \sum_{k=0}^3 \Phi_M(j,k) \mG_k(\t) \qquad\mbox{for all } M \in \GG_0(2),
\end{equation*}
where the multiplier system $\Phi$ is as follows for $T=\pmat{1&1\\0&1}$ and $R=\pmat{1&0\\2&1}$:
\begin{equation*}
\Phi_T =  \mat{1&0&0&0\\0&\z_6&0&0\\0&0&\z_3^2&0\\0&0&0&-1}\andd\Phi_R = \frac{1}{\sqrt{3}} \mat{0&2\z_{12}&0&\z_4^3\\\z_{12}&0&\z^{11}_{12}&0\\0&\z^{11}_{12}&0&\z_{12}\\\z_4^3&0&2\z_{12}&0}.
\end{equation*}
\end{prop}

\subsection{Quantum modularity}

In summary, from Proposition \ref{prop:L_5_8_modular_transformation} we know that $\mG_j$ for $j\in\{0,1,2,3\}$ form a vector-valued Maass waveform for $\GG_0(2)$. We next apply the findings of Section \ref{sec:maass_forms} on these functions. First, we note that equation \eqref{eq:defn_u_and_error} and Proposition \ref{prop:uj_fourier_expansion} imply that the functions $\mg_j$ (which are basically equal to the functions $L_5,\dots,L_8$ thanks to Lemma \ref{lem:l5l8rew}) are related to the $\mG_j$'s as
\begin{equation*}
\mg_j(\t) = -\frac2\pi\int_\t^{i\infty} [\mG_j(z),R_\t(z)].
\end{equation*}
Then Proposition \ref{prop:uj_modularity} gives their modular transformation for any $M:=\pmat{a&b\\c&d}\in\GG_0(2)$ and $\t\in\H$ with $\t_1\ne-\frac dc$:
\begin{equation*}
\mg_j\lp \frac{a\t+b}{c\t+d}\rp 
= \sgn(c\t_1+d) (c\t+d)  \sum_{k=0}^3 \Phi_M(j,k) 
\lp \mg_k(\t)+\mcalG_{k,-\frac{d}{c}}(\t)\rp.
\end{equation*}
Here the obstruction to modularity $\mcalG_{k,-\frac{d}{c}}$ is a holomorphic function on $\C \setminus ( -\frac{d}{c} + i \R )$ (according to Proposition \ref{prop:modularity_error_holomorphy}) defined as in equation \eqref{eq:defn_u_and_error}:
\begin{equation*}
\mcalG_{j,\varrho}(\t) := \frac2\pi\int_{\varrho}^{i\infty} [\mG_j(z),R_\t(z)] .
\end{equation*}

Next we follow Proposition \ref{prop:uj_vertical_limit} and define the functions $\mfrakg_j : \calQ_{\GG_0 (2)} \to \C$ with $j\in\{0,1,2,3\}$ by 
\begin{equation}\label{eq:L5_8_quantum_limit_defn}
\mfrakg_j(x) := \lim_{t\to0^+} \lp \mg_j(x+it) - \frac{2 \arctanh \lp \frac{1}{\sqrt{3}} \rp}{\pi |c_x| t}  \Phi_{M_x^{-1}}(j,0) 
\rp ,
\end{equation}
where $x=-\frac{d_x}{c_x}$ and $M_x=\pmat{a_x&b_x\\c_x&d_x}\in\GG_0 (2)$.
Then finally by Theorem \ref{thm:frak_uj_quantum_modular_transformations} we find the quantum modular properties for (the finite parts of) the rational limits of $L_5, \ldots, L_8$.
\begin{prop}\label{prop:L5_8_quantum_modular}
The functions $\mfrakg_j : \calQ_{\GG_0 (2)} \to \C$ with $j\in\{0,1,2,3\}$ form a vector-valued quantum modular form transforming as follows for any $M := \pmat{a&b\\c&d}\in \GG_0 (2)$ and $x \in \mathcal{Q}_{\GG_0 (2)} \setminus \{-\frac{d}{c} \}$:
\begin{equation*}
\mfrakg_j\lp \frac{ax+b}{cx+d}\rp 
= |cx+d| \sum_{k=0}^3 \Phi_M(j,k) 
\lp \mfrakg_k(x)+\mcalG_{k,-\frac{d}{c}}(x)\rp .
\end{equation*}
Here the multiplier system $\Phi$ is as given in Proposition \ref{prop:L_5_8_modular_transformation}.
\end{prop}

\section{The functions $L_9$\textendash$L_{12}$}\label{sec:L9_L12_analysis}

In this section we study our final family of functions from \cite{LoOs} with $L_9,\dots,L_{12}$ that count ideals in the ring of integers of $\Q(\sqrt{6})$.

\subsection{Rewriting as theta functions}

We begin like the previous two families of functions and rewrite $L_9,\dots,L_{12}$ in terms of theta functions.

\begin{lem}\label{lem:l9l12}
We have
\begin{align*}
q^{-\frac{9}{16}}L_9(q) &= \frac12\left(\sum_{\bm n\in\Z^2+\left(0\ \frac38\right)^T} + \sum_{\bm n\in\Z^2+\left(\frac12\ \frac78\right)^T}\right) (1+\sgn(n_1+n_2)\sgn(n_1-n_2)) q^{6n_1^2-4n_2^2},\\
q^{-\frac{17}{16}}L_{10}(q) &= \frac12\left(\sum_{\bm n\in\Z^2+\left(0\ \frac18\right)^T} + \sum_{\bm n\in\Z^2+\left(\frac12\ \frac58\right)^T}\right) (1+\sgn(n_1+n_2)\sgn(n_1-n_2)) q^{6n_1^2-4n_2^2},\\
q^{\frac{5}{48}}L_{11}(q) &= \frac12\left(\sum_{\bm n\in\Z^2+\left(\frac16\ \frac18\right)^T} + \sum_{\bm n\in\Z^2+\left(\frac23\ \frac58\right)^T}\right) (1+\sgn(n_1+n_2)\sgn(n_1-n_2)) q^{6n_1^2-4n_2^2},\\
q^{-\frac{19}{48}}L_{12}(q) &= \frac12\left(\sum_{\bm n\in\Z^2+\left(\frac16\ \frac38\right)^T} + \sum_{\bm n\in\Z^2+\left(\frac23\ \frac78\right)^T}\right) (1+\sgn(n_1+n_2)\sgn(n_1-n_2)) q^{6n_1^2-4n_2^2}.
\end{align*}
\end{lem}

\begin{proof}
Again the proof is similar to that of Lemma \ref{lem:l1l4} and the identities we require are taken from the proof of Theorem 1.3 of \cite{Lo}
\begin{multline*}
q^{-9}L_9\left(q^{16}\right) = \sum_{\substack{n\ge1\\-n\le j\le n-1}} q^{6(4n)^2-(8j+3)^2} + \sum_{\substack{n\ge0\\-n\le j\le n}} q^{6(4n+2)^2-(8j+1)^2}\\
+ \sum_{\substack{n\ge1\\-n+1\le j\le n}} q^{6(4n)^2-(8j-3)^2} + \sum_{\substack{n\ge0\\-n\le j\le n}} q^{6(4n+2)^2-(8j-1)^2},
\end{multline*}
\begin{multline*}
q^{-17}L_{10}\left(q^{16}\right) = \sum_{\substack{n\ge1\\-n\le j\le n-1}} q^{6(4n)^2-(8j+1)^2} + \sum_{\substack{n\ge0\\-n\le j\le n}} q^{6(4n+2)^2-(8j+3)^2}\\
+\sum_{\substack{n\ge1\\-n+1\le j\le n}} q^{6(4n)^2-(8j-1)^2} + \sum_{\substack{n\ge0\\-n\le j\le n}} q^{6(4n+2)^2-(8j-3)^2},
\end{multline*}
\begin{multline*}
q^{10}L_{11}\left(q^{96}\right) = \sum_{\substack{n\ge0\\-n\le j\le n}} \left(q^{(24n+4)^2-6(8j+1)^2}+q^{(24n+20)^2-6(8j+1)^2}\right)\\
+\sum_{\substack{n\ge0\\-n-1\le j\le n}} \left(q^{(24n+32)^2-6(8j+3)^2}+q^{(24n+16)^2-6(8j+3)^2}\right),
\end{multline*}
\begin{multline*}
q^{-38}L_{12}\left(q^{96}\right) = \sum_{\substack{n\ge1\\-n\le j\le n-1}} q^{(24n+4)^2-6(8j+3)^2} + \sum_{\substack{n\ge0\\-n-1\le j\le n}} q^{(24n+20)^2-6(8j+3)^2}\\
+\sum_{\substack{n\ge-1\\-n-1\le j\le n+1}} q^{(24n+32)^2-6(8j+1)^2} + \sum_{\substack{n\ge0\\-n\le j\le n}} q^{(24n+16)^2-6(8j+1)^2}. \qedhere
\end{multline*}
\end{proof}
For more compact expressions, we 
define the quadratic form
\begin{equation}\label{eq:quadratic_form_example3}
Q(\bm n):=6n_1^2-4n_2^2
\end{equation}
and the vectors 
\begin{equation}\label{eq:cj_vector_example3}
\bm{c_1}:=\frac{1}{2\sqrt{3}}(-2,3)^T
\andd
\bm{c_2}:=\frac{1}{2\sqrt{3}}(2,3)^T
\end{equation}
in $C_Q$ (as defined in Remark \ref{rem:P_choice}). Then the theta functions in \eqref{eq:defn_mock_maass_q_series} are
\begin{align*}\label{eq:vartheta_example3}
\mt_{\bm{\mu}}(\tau)
=&
\frac12 \sum_{\bm n \in \Z^2 + \bm \mu} 
\lp 1 + \sgn(n_1+n_2) \sgn(n_1-n_2) \rp q^{6n_1^2 - 4n_2^2}
\quad
\mbox{for } \bm{\mu} \not\in \Z^2.
\end{align*}
With this expression at hand, we can rewrite the results in Lemma \ref{lem:l9l12} as follows.
\begin{lem}\label{lem:l9l12rew}
We have 
\begin{equation*}
q^{-\frac{9}{16}} L_9 (q) = \mh_3 (\t),  \qquad
q^{-\frac{17}{16}} L_{10} (q)  = \mh_0 (\t),\qquad
q^{\frac{5}{48}} L_{11} (q) = \mh_1 (\t),\qquad
q^{-\frac{19}{48}} L_{12} (q)  = \mh_2 (\t),
\end{equation*}
where
\begin{equation*}
\mh_j (\t) := \mt_{\bm{\mu_j}}(\t) + \mt_{\bm{\mu_j}+\bm\l}(\t) 
\quad \mbox{with }
\bm{\mu_j}:=\lp \frac j6\ \ \frac18 \rp^T \mbox{ and } 
\bm\l:=\lp \frac12\ \ \frac12 \rp^T.
\end{equation*}
\end{lem}

\subsection{The corresponding Maass forms}

Our next step is to define the mock Maass theta function $\mT_{\bm\mu}$ as in equation \eqref{eq:defn_mock_maass_theta_fnc} for the quadratic form $Q$ in \eqref{eq:quadratic_form_example3} and vectors $\bm{c_1}$, $\bm{c_2}$ in \eqref{eq:cj_vector_example3}. We first find the vectors in $C_Q^\perp$ that correspond to $\bm{c_1}$ and $\bm{c_2}$ as $\bm{c_1^\perp}=\frac{1}{\sqrt{2}}(1\ -1)^T$ and $\bm{c_2^\perp}=\frac{1}{\sqrt{2}}(1\ \ 1)^T$ to get (for $\bm{\mu} \not\in \Z^2$)
\begin{align*}
\mT_{\bm\mu}(\t) &= \frac{\sqrt{\t_2}}{2} \sum_{\bm n\in\Z^2+\bm\mu} \lp 1 + \sgn(n_1+n_2) \sgn(n_1-n_2) \rp K_0\left(2\pi \left(6n_1^2 - 4n_2^2\right)\tau_2\right) e^{2\pi i \left(6n_1^2 - 4n_2^2\right)\tau_1}\\
&+\frac{\sqrt{\t_2}}{2} \sum_{\bm n \in \Z^2 + \bm{\mu}} \lp 1 - \sgn(3n_1+2n_2) \sgn(3n_1-2n_2) \rp K_0\left(-2\pi (6n_1^2 - 4n_2^2) \tau_2\right) e^{2\pi i \left(6n_1^2 - 4n_2^2\right) \tau_1} .
\end{align*}
These mock Maass theta functions have modular completions,
\begin{equation*}
\wh{\mT}_{\bm\mu} = \mT_{\bm\mu}+ \varphi_{\bm\mu}^{[\bm{c_1}]}  - \varphi_{\bm\mu}^{[\bm{c_2}]},
\end{equation*}
as described in equation \eqref{eq:mock_maass_theta_fnc_completion} and Proposition \ref{prop:quadzero}.
Like the other two cases, these functions satisfy the following elementary properties:
\begin{equation*}\label{eq:F_L9_12_identifications}
\mT_{\bm\mu} = \mT_{-\bm\mu} = \mT_{{(-\mu_1\ \mu_2)^T}} = \mT_{{(\mu_1\ -\mu_2)^T}},\qquad \wh{\mT}_{\bm\mu} = \wh{\mT}_{-\bm\mu} = \wh{\mT}_{{(-\mu_1\ \mu_2)^T}} = \wh{\mT}_{{(\mu_1\ -\mu_2)^T}}.
\end{equation*}
Since we would like to study the linear combinations of $\mt_{\bm\mu}$'s given in Lemma \ref{lem:l5l8rew}, we define the following for $j\in\{0,1,2,3\}$:
\begin{equation*}
\mH_j  := \mT_{\bm{\mu_j}} +  \mT_{\bm{\mu_j}+\bm{\l}} 
\andd
\wh{\mH}_j  := \wh{\mT}_{\bm{\mu_j}} +  \wh{\mT}_{\bm{\mu_j}+\bm{\l}} .
\end{equation*}
The shadow contributions to the $\wh{\mH}_j$'s vanish following the arguments of Lemma \ref{lem:modular_obstruction_cancellation}.
\begin{lem}\label{lem:whH}
For $j \in \{0,1,2,3\}$ we have
\begin{equation*}
\mH_j  = \wh{\mH}_j  .
\end{equation*}
\end{lem}
Following the proof of Proposition \ref{prop:L_1_4_modular_transformation}, the $\mH_j$'s have the following modular transformations.
\begin{prop}\label{prop:L_9_12_modular_transformation}
The functions $\mH_j$ for $j\in\{0,1,2,3\}$ transform like vector-valued modular function under $\GG_0(2)$:
\begin{equation*}
\mH_j\left(\frac{a\t+b}{c\t+d}\right) = \sum_{k=0}^3 \OO_M(j,k) \mH_k(\t) \qquad \mbox{for all } M \in \GG_0(2),
\end{equation*}
where the multiplier system $\OO$ is as follows for $T=\pmat{1&1\\0&1}$ and $R=\pmat{1&0\\2&1}$:
\begin{align*}
\OO_T &= \mat{\z^{15}_{16}&0&0&0\\0&\z^5_{48}&0&0\\0&0&\z^{29}_{48}&0\\0&0&0&\z^7_{16}},\\
\OO_R &= \sqrt{\frac{2-\sqrt{2}}{12}} \mat{\z^{15}_{16}&2\left(1+\sqrt{2}\right)\z_{48}&2\z^{13}_{48}&\left(1+\sqrt{2}\right)\z^{11}_{16}\\ \left(1+\sqrt{2}\right)\z_{48}&\z^5_{48}&\left(1+\sqrt{2}\right)\z^{41}_{48}&\z^{13}_{48}\\ \z^{13}_{48}&\left(1+\sqrt{2}\right)\z^{41}_{48}&\z^5_{48}&\left(1+\sqrt{2}\right)\z_{48}\\ \left(1+\sqrt{2}\right)\z^{11}_{16}&2\z^{13}_{48}&2\left(1+\sqrt{2}\right)\z_{48}&\z^{15}_{16}}.
\end{align*}
\end{prop}

\subsection{Quantum modularity}

Summarizing the results above, we find from Proposition \ref{prop:L_5_8_modular_transformation} that $\mH_j$ for $j\in\{0,1,2,3\}$ form a vector-valued Maass form for $\GG_0(2)$. Our next step is to apply the results of Section \ref{sec:maass_forms}. First, we note that \eqref{eq:defn_u_and_error} and Proposition \ref{prop:uj_fourier_expansion} imply that the functions $\mh_j$ (which are basically equal to $L_9,\dots,L_{12}$ thanks to Lemma \ref{lem:l9l12rew}) are related to $\mH_j$'s as
\begin{equation*}
\mh_j(\t) = -\frac2\pi\int_\t^{i\infty} [\mH_j(z),R_\t(z)].
\end{equation*}
Then Proposition \ref{prop:uj_modularity} gives their modular transformation for any $M:=\pmat{a&b\\c&d}\in\GG_0 (2)$ and $\t\in\H$ with $\t_1\ne-\frac dc$ as
\begin{equation*}
\mh_j\lp \frac{a\t+b}{c\t+d}\rp 
= \sgn(c\t_1+d) (c\t+d)  \sum_{k=0}^3 \Omega_M(j,k) 
\lp \mh_k(\t)+\mcalH_{k,-\frac{d}{c}}(\t)\rp .
\end{equation*}
Here the obstruction to modularity $\mcalH_{k,-\frac{d}{c}}$ is a holomorphic function on $\C \setminus ( -\frac{d}{c} + i \R )$ (according to Proposition \ref{prop:modularity_error_holomorphy}) defined as in equation \eqref{eq:defn_u_and_error}:
\begin{equation*}
\mcalH_{j,\varrho}(\t) := \frac2\pi\int_{\varrho}^{i\infty} [\mH_j(z),R_\t(z)] .
\end{equation*}

Next we follow Proposition \ref{prop:uj_vertical_limit} and define the functions $\mfrakh_j : \calQ_{\GG_0 (2)} \to \C$ with $j\in\{0,1,2,3\}$ by 
\begin{equation}\label{eq:L9_12_quantum_limit_defn}
\mfrakh_j(x) := \lim_{t\to0^+} \mh_j(x+it) .
\end{equation}
Then Theorem \ref{thm:frak_uj_quantum_modular_transformations} implies that these functions (given by rational limits of $L_9, \ldots, L_{12}$) form a quantum modular form.
\begin{prop}\label{prop:L9_12_quantum_modular}
The functions $\mfrakh_j : \calQ_{\GG_0 (2)} \to \C$ with $j\in\{0,1,2,3\}$ form a vector-valued quantum modular form transforming as follows for any $M := \pmat{a&b\\c&d}\in \GG_0 (2)$ and $x \in \mathcal{Q}_{\GG_0 (2)} \setminus \{-\frac{d}{c} \}$:
\begin{equation*}
\mfrakh_j\lp \frac{ax+b}{cx+d}\rp 
= |cx+d| \sum_{k=0}^3 \Omega_M(j,k) 
\lp \mfrakh_k(x)+\mcalH_{k,-\frac{d}{c}}(x)\rp .
\end{equation*}
Here the multiplier system $\Phi$ is as given in Proposition \ref{prop:L_9_12_modular_transformation}.
\end{prop}

\section*{Appendix: Numerical Examples}

In this appendix, we provide some numeric results on the functions $\mg_j$ and quantum modular forms $\mfrakg_j$ associated to the family $L_5 (q), \ldots, L_8 (q)$. This example is distinguished from the other two by the presence of a nonzero constant term for the Maass form $\mG_j $. Due to this property, for a given $x \in \calQ_{\GG_0 (2)}$ not all components of $\mg_j (x + it)$ converge as $t \to 0^+$. The quantum modular form $\mfrakg_j$ is then defined by simply removing the leading growing term from this object. With the results below we try to exemplify various aspects of these statements. 

Firstly, in Table \ref{tab:g_values} we give the approximate values of $\mg_j (\t)$, $\mg_j (\frac{\t}{2 \t+1})$, and  $\calG_{j,-\frac{1}{2}} (\t)$ for various values of $\t$ that get close to $\frac{11}{12} \in \calQ_{\GG_0 (2)}$. With these numbers one can check that the modular transformation property in Proposition \ref{prop:L_5_8_modular_transformation} is satisfied to the order shown.
\begin{table}[h]
\centering
\begin{tabular}{c | c c c}
\toprule
$\t$  & $\mg_j (\t)$ & $\mg_j( \frac{\t}{2 \t+1} )$ & $\calG_{j,-\frac{1}{2}} (\t)$\\[2pt]
\hline \vspace{-3mm}
\\[0pt]
$\frac{11}{12} + \frac{i}{10^2}$ & 
$\pmat{
2.20152385 - 1.72453350 i \\
0.60261342 -1.52382165 i \\
-1.24906564 +0.76567897 i \\
-0.69157228 +1.38696416i
}$
&
$\pmat{
5.73980919 -2.45135588 i \\
4.58641591+1.38204971i\\
-2.75333449-0.96881340i \\
-8.20830829-3.93452617i
}$
&
$\pmat{
0.48547906-0.41384103i \\
-0.23702464+0.05235489i \\
0.07430784+0.09577127i\\
0.03174146+0.00641302i 
}$
\\[15pt]
$\frac{11}{12} + \frac{i}{10^3}$ & 
$\pmat{
34.06349943-1.54026658i \\
-0.63001497-4.14866463i \\
-1.65163852-0.02850558i \\
-1.30982870+1.39636988i 
}$
&
$\pmat{
6.54926927-10.93296376i \\
48.34343745+26.91081033i \\
-7.53508162-4.15860747i \\
-7.73554006-58.91067177i 
}$
&
$\pmat{
0.48541322-0.41343939i \\
-0.23705194+0.05106932i \\
0.07383483+0.09621073i \\
0.03179379+0.00676626i
}$
\\[15pt]
$\frac{11}{12} + \frac{i}{10^4}$ & 
$\pmat{
348.48937661-1.53336988i \\
-0.32418957-3.49803282i \\
-1.72369780-0.13566295i \\
-1.40302668+1.41219531i
}$
&
$\pmat{
6.37047977-8.44131747i \\
493.58357944+284.00154421i \\
-6.71745517-3.53610552i \\
-7.75436167-573.67311441i
}$
&
$\pmat{
0.48540637-0.41339930i \\
-0.23705391+0.05094084i \\
0.07378727+0.09625443i \\
0.03179870+0.00680169i
}$
\\[15pt]
$\frac{11}{12} + \frac{i}{10^5}$ & 
$\pmat{
3492.49719898-1.53275909i \\
-0.30408895-3.45551296i \\
-1.73170057-0.14713202i \\
-1.41308637+1.41400914i
}$
&
$\pmat{
6.36033233-8.27194978i \\
4947.57332454+2855.51730098i \\
-6.67017442-3.49755349i \\
-7.75743061-5716.76719430i
}$
&
$\pmat{
0.48540568-0.41339529i \\
-0.23705410+0.05092799i \\
0.07378251+0.09625879i \\
0.03179919+0.00680523i
}$
\\
\bottomrule
\end{tabular}
\vspace{-7pt}
\caption{}
\label{tab:g_values}
\vspace{-0pt}
\end{table}

Now we note that for the point $x=\frac{11}{12}$ we have
\begin{equation*}
\g_{0,x} = 0.10974649141040139\dots
\end{equation*}
with all the other components zero. So $\mg_0(x+it)$ is the only component of $\mg$ that diverges as $t\to0^+$. This is already visible in Table \ref{tab:g_values}. In fact, subtracting the growing piece as in Proposition \ref{prop:uj_vertical_limit} we find the results in Table \ref{tab:g_values_limiting}.

{
\begin{table}[h]
\centering
\begin{tabular}{c | c}
\toprule
$\t$  & $\mg_0 (\t) - \frac{1}{\pi \t_2} \g_{0,\frac{11}{12}}$
\\ [5pt]
\hline\vspace{-3mm}
\\
$\frac{11}{12} + \frac{i}{10^2}$ & 
$-1.2918154651\ldots-1.7245334957\ldots i$
\\[5pt]
$\frac{11}{12} + \frac{i}{10^3}$ & 
$-0.8698937600\ldots-1.5402665800\ldots i$
\\[5pt]
$\frac{11}{12} + \frac{i}{10^4}$ & 
$-0.8445552866\ldots-1.5333698829\ldots i$
\\[5pt]
$\frac{11}{12} + \frac{i}{10^5}$ & 
$-0.8421200156\ldots-1.5327590942\ldots i$
\\
\bottomrule
\end{tabular}
\vspace{-7pt}
\caption{}
\label{tab:g_values_limiting}
\vspace{-8pt}
\end{table}
}
Given the integral representation of $\mg_j$ in \eqref{eq:defn_u_and_error} and the modular transformations of the Maass form $\mG_j$, one can efficiently compute the values of the quantum modular form $\mfrakg_j$. In particular, we have
\begin{equation*}
\mfrakg_0 \lp \frac{11}{12} \rp = 
-0.8418504490893569688\ldots 
-1.532692070451105313\ldots i
\end{equation*}
One can check the quantum modular transformations using the approximate values in Table \ref{tab:g_values_quantum}.

{
\begin{table}[h]
\centering
\begin{tabular}{c c c}
\toprule
$\mfrakg_j \left(\frac{11}{12}\right)$ & $\mfrakg_j \left(\frac{11}{34}\right)$& $\calG_{j,-\frac{1}{2}} \left(\frac{11}{12}\right)$
\\ 
\hline \vspace{-3mm}
\\
$\mat{
-0.84185045 -1.53269207i \\
-0.30191635-3.45091966i \\
-1.73260053-0.14841718i \\
-1.41421356+1.41421356i
}$
&
$\mat{
6.35925312-8.25358660i \\
-1.30588754-1.76561858i \\
-6.66512984-3.49340997i \\
-7.79812377-2.27823717i
}$
&
$\mat{
0.48540561-0.41339484i \\
-0.23705412+0.05092656i\\
0.07378198+0.09625928i \\
0.03179924+0.00680562i 
}$
\\
\bottomrule
\end{tabular}
\vspace{-7pt}
\caption{}
\label{tab:g_values_quantum}
\vspace{-8pt}
\end{table}
}

Finally, in Figures \ref{fig:quantum_1} and \ref{fig:quantum_1_transf} we display various values of $\mfrakg_0$ and its obstruction to modularity.

\begin{figure}[h!]
 \vspace{-10pt}
  \centering
    \includegraphics[scale=0.25]{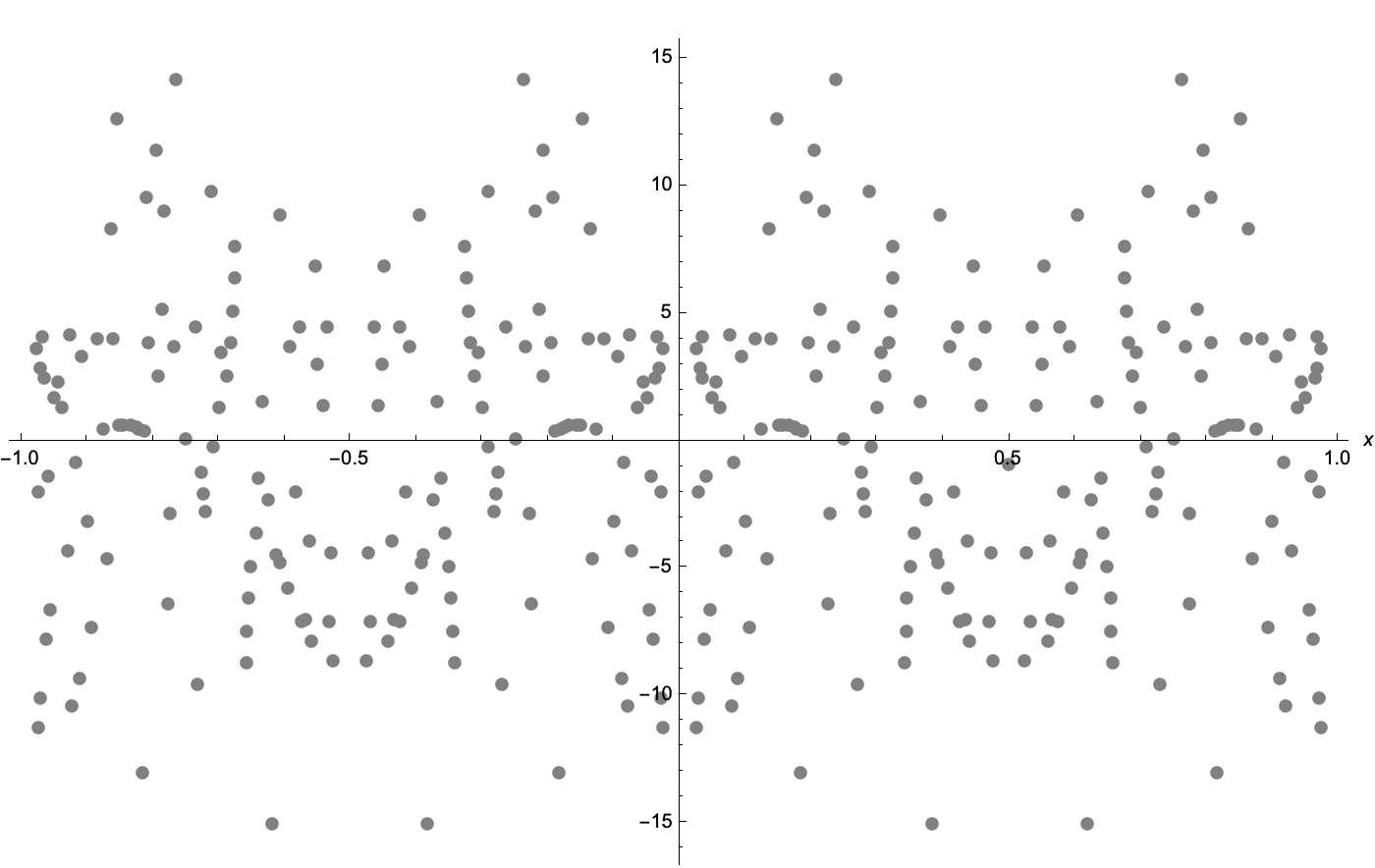}
    \hspace{40pt}
    \includegraphics[scale=0.25]{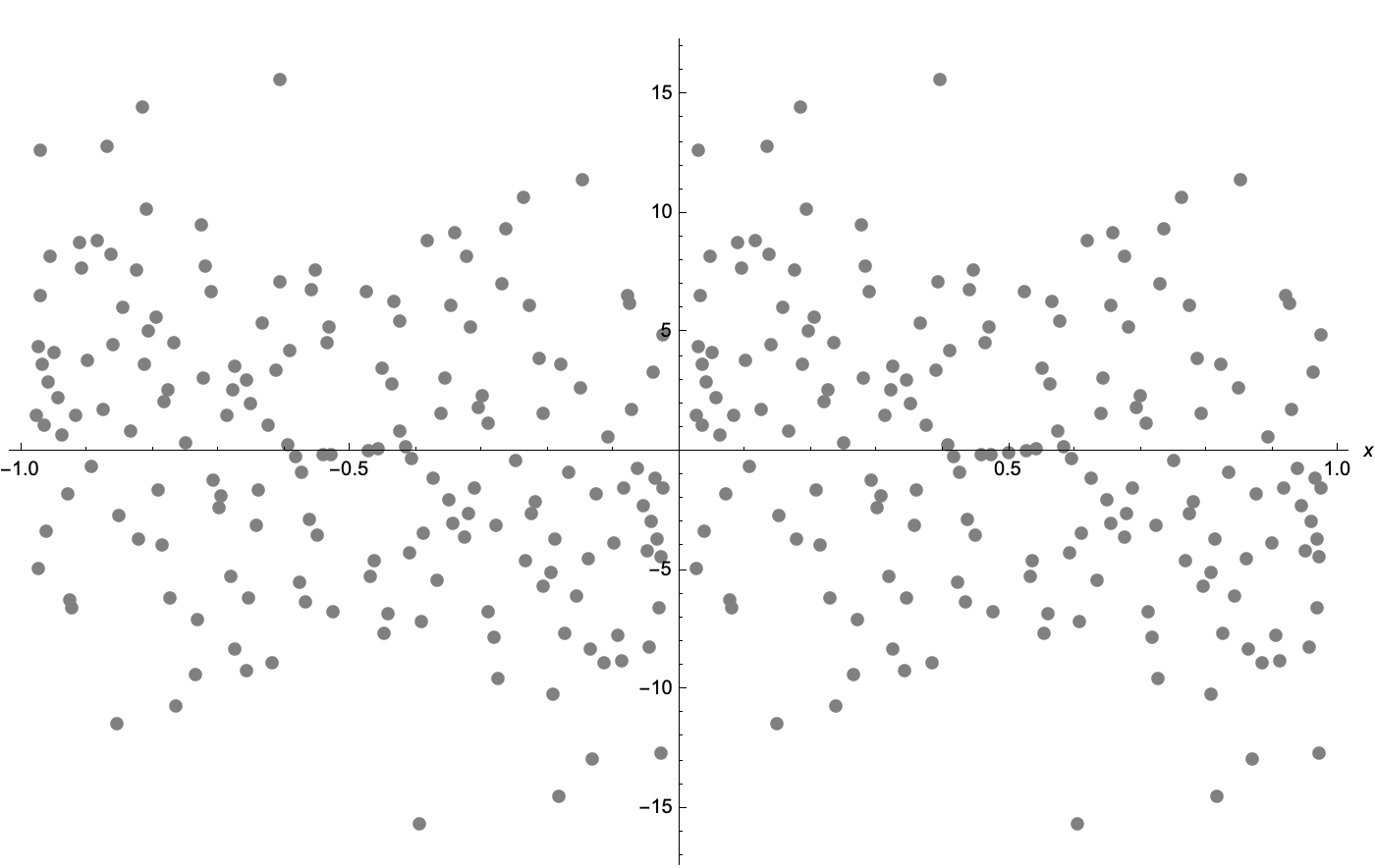}
     \vspace{-10pt}
    \caption{On the left we plot $\mathrm{Re} (\mfrakg_0 (x) )$ and on the right we plot $\mathrm{Im} ( \mfrakg_0 (x) )$ for $x \in \mathcal{Q}_{\GG_0 (2)}$ with $-1 < x < 1$ and denominator at most $40$.}
    \label{fig:quantum_1}
\end{figure}

\begin{figure}[h!]
 \vspace{-10pt}
  \centering
    \includegraphics[scale=0.25]{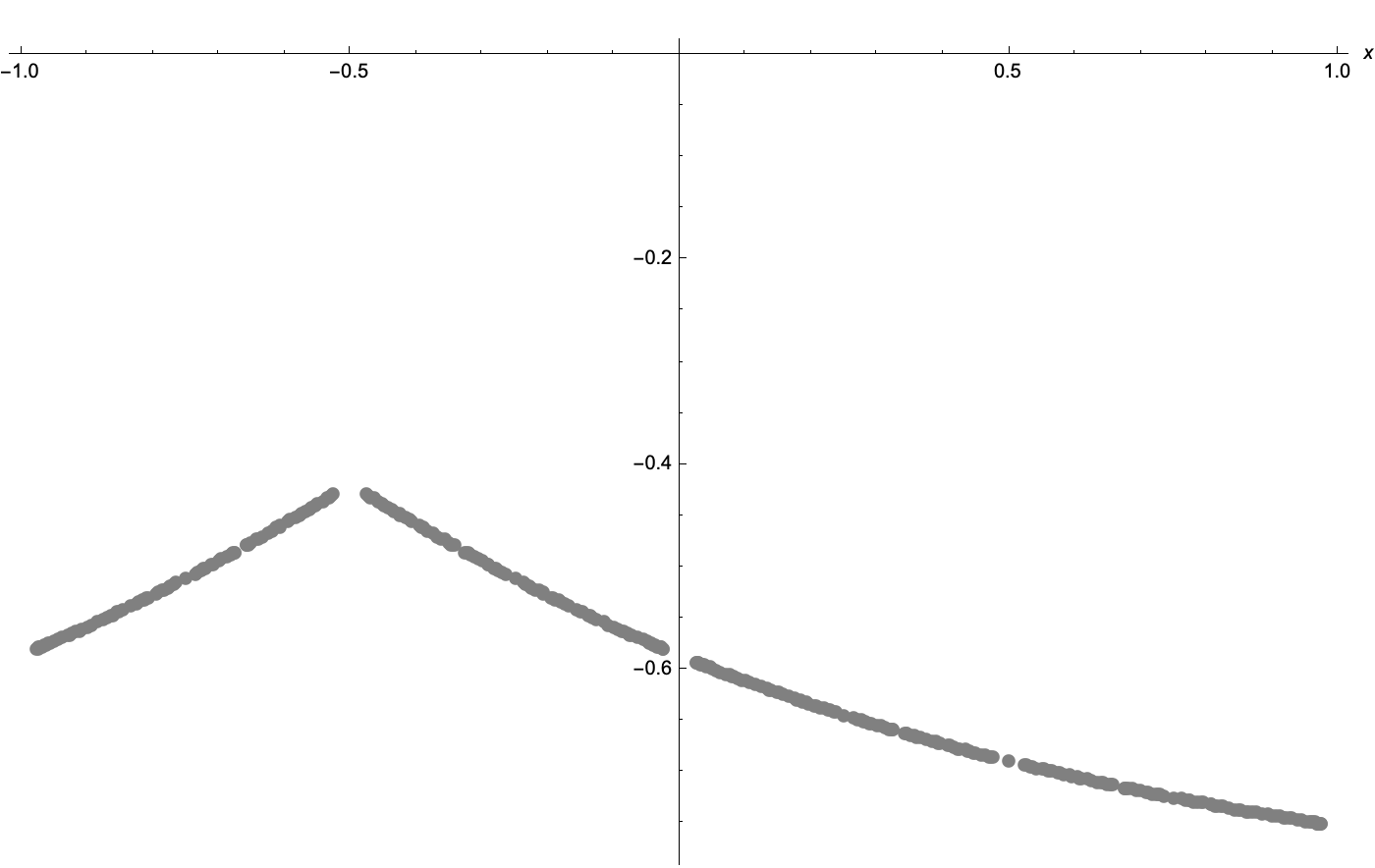}
    \hspace{40pt}
    \includegraphics[scale=0.25]{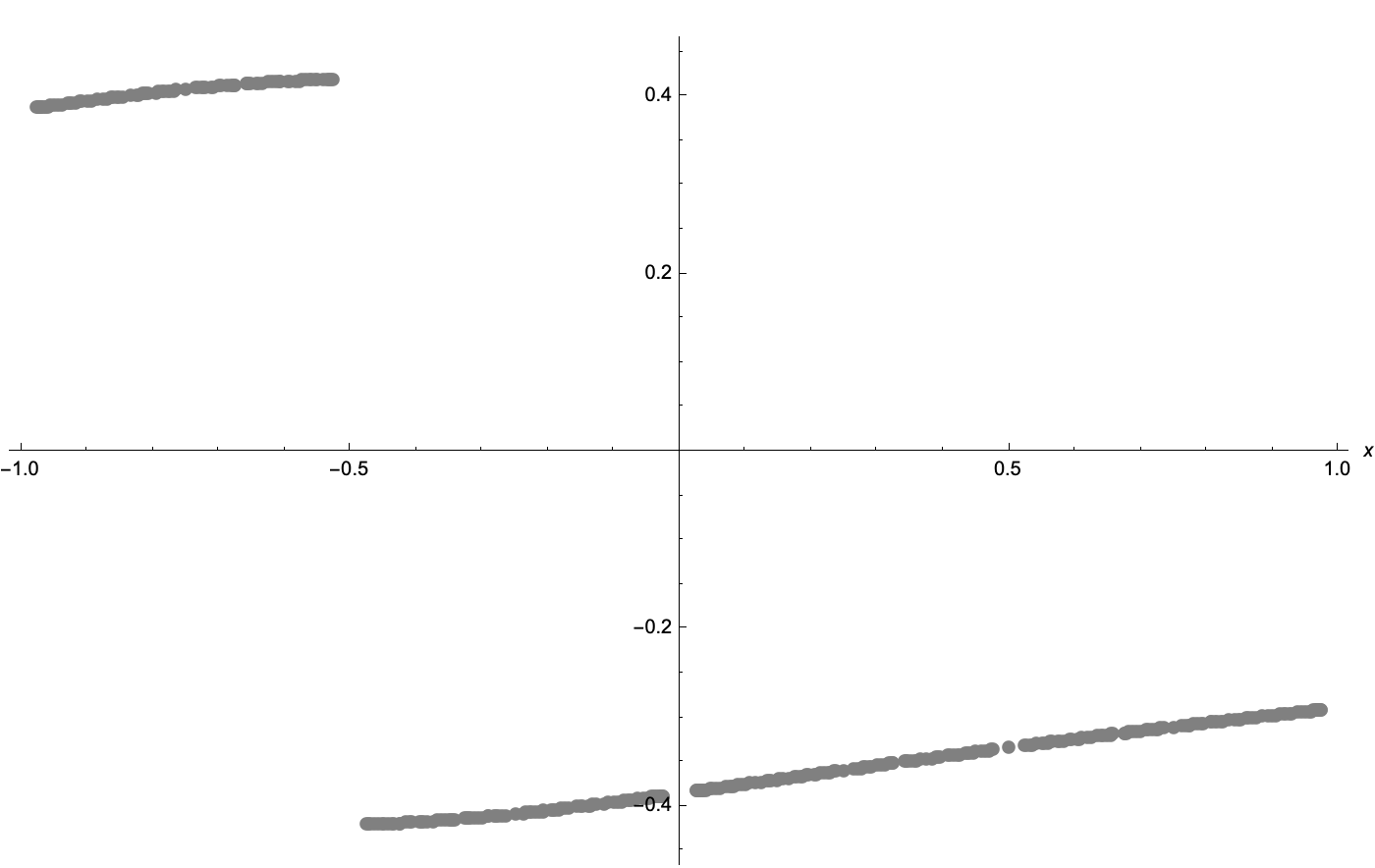}
     \vspace{-10pt}
    \caption{On the left we plot the real part and on the right we plot the imaginary part of $\mfrakg_0 ( \frac{x}{2x+1}) - |cx+d| \sum_{k=0}^3 \Phi_R(0,k) \mfrakg_k(x)$ for $x \in \mathcal{Q}_{\GG_0 (2)}$ with $x\in(-1,1)$ and denominator at most $40$.}
    \label{fig:quantum_1_transf}
\end{figure}

\end{document}